\documentclass[preprint,12pt]{elsarticle}

\usepackage[T1]{fontenc}
\usepackage{geometry}
\usepackage{graphicx}%
\usepackage{amsmath,amssymb,amsfonts}%
\usepackage{amsthm}%
\usepackage{mathrsfs}%
\usepackage[title]{appendix}%
\usepackage{xcolor}%
\usepackage{textcomp}%
\usepackage{manyfoot}%
\usepackage{booktabs}%
\usepackage{algorithm}%
\usepackage{algorithmicx}%
\usepackage{algpseudocode}%
\usepackage{listings}%
\usepackage{indentfirst}
\usepackage{lineno}
\usepackage{mathrsfs}
\usepackage{array}
\usepackage{float}
\usepackage{subfigure}
\usepackage[justification=centering]{caption}
\usepackage{verbatim}
\usepackage{setspace}
\usepackage{makecell}
\usepackage{textcomp}
\usepackage{multirow}%

\usepackage{bm}
\usepackage{url}
\usepackage{hyperref}
\usepackage{marvosym}
\usepackage{cleveref}

\newcommand{\red}[1]{\textcolor[rgb]{1.00,0.00,0.00}{#1}}

\newcommand{\Rmnum}[1]{\expandafter\@slowromancap\romannumeral #1@}
\newcommand{\tensor}[1]{\bm{\mathcal{#1}}}
\newcommand{\norm}[2]{\left\| #1 \right\|_{#2}}

\numberwithin{equation}{section}
\newtheorem{thm}{Theorem}

\newtheorem{rmk}{Remark}

\begin{document}

\begin{frontmatter}



\title{APTT: An accuracy-preserved tensor-train method for the Boltzmann-BGK equation}


 \author[mymainaddress]{Zhitao Zhu}
\ead{zhuzt@pku.edu.cn}

\author[mymainaddress,myfirstaddress]{Chuanfu Xiao\textsuperscript{\Letter}}
	\ead{chuanfuxiao@pku.edu.cn}

 \author[myfirstaddress]{Kejun Tang}
\ead{tangkejun@icode.pku.edu.cn}

 \author[mysecondaddress]{Jizu Huang}
\ead{huangjz@lsec.cc.ac.cn}

 \author[mymainaddress,myfirstaddress]{Chao Yang}
	\ead{chao\_yang@pku.edu.cn}

 \address[mymainaddress]{School of Mathematical Sciences, Peking University, Beijing 100871, China}
 \address[myfirstaddress]{PKU-Changsha Institute for Computing and Digital Economy, Changsha 410205, China}
 \address[mysecondaddress]{LSEC, Institute of Computational Mathematics, Academy of Mathematics and System Sciences, Chinese Academy of Sciences, Beijing 100190, China}



\begin{abstract}
Solving the Boltzmann-BGK equation with traditional numerical methods suffers from high computational and memory costs due to the curse of dimensionality.
In this paper, we propose a novel accuracy-preserved tensor-train (APTT) method to efficiently solve the Boltzmann-BGK equation.
A second-order finite difference scheme is applied to discretize the Boltzmann-BGK equation, resulting in a tensor algebraic system at each time step.
Based on the low-rank TT representation, the tensor algebraic system is then approximated as a TT-based low-rank system, which is efficiently solved using the TT-modified alternating least-squares (TT-MALS) solver.
Thanks to the low-rank TT representation, the APTT method can significantly reduce the computational and memory costs compared to traditional numerical methods.
Theoretical analysis demonstrates that the APTT method maintains the same convergence rate as that of the finite difference scheme.
The convergence rate and efficiency of the APTT method are validated by several benchmark test cases.

\end{abstract}




\end{frontmatter}




\section{Introduction}\label{sec_intro}

The Boltzmann transport equation (BTE)  is widely used to model super- and hyper-sonic flows \cite{boyd1995predicting}, the flow of electrons in metals and silicon \cite{sondheimer2001mean}, microflow systems \cite{karniadakis2002micro}, and rarefied gas flows \cite{porodnov1974experimental}. Due to its high dimensionality and the intricate quadratic integral collision with a singular collision kernel, solving the BTE efficiently and accurately presents a significant challenge.
As a reduced order model of the complex collision operator, the widely used collision operator is a linear relaxation collision term known as the Bhatnagar–Gross–Krook (BGK) approximation to simplify calculations while ensuring compliance with conservation laws \cite{bhatnagar1954model}. This paper specifically focuses on the Boltzmann-BGK equation, a variant of the BTE under the BGK approximation.

Since the Boltzmann-BGK equation is a high-dimensional partial differential equation (PDE), using classical numerical methods (e.g., the finite element method or the finite difference method) to solve the Boltzmann-BGK equation suffers from high computational cost when the fine mesh is employed. As a mesh-free approach, neural network methods are proposed to solve such challenging problems \cite{de2021physics, lou2021physics, li2024solving}. However, for solving PDEs, the neural network method usually cannot provide rigorous guarantees of accuracy and robustness \cite{karniadakis2021physics, hao2022physics}.
In recent years, low-rank tensor methods have gained significant attention for solving high-dimensional partial differential equations (PDEs) \cite{khoromskij2015tensor,bachmayr2016tensor,khoromskij2018tensor,bachmayr2023low}. This is because tensor methods can provide an efficient way to solve the large-scale linear system arising from PDEs with guaranteed accuracy, while it is a challenging task for traditional numerical methods. On the one hand, the tensor methods are usually used as an algebraic solver for classic discrete systems. Thereby, tensor methods maintain the advantages of traditional numerical methods in terms of high accuracy and preservation of conservation laws. On the other hand, due to the low-rank structures of the target solutions and low-rank approximation techniques, the high computational and memory costs are reduced dramatically in tensor methods. For example,
low-rank tensor approaches have successfully been applied to various PDEs, including the Vlasov equation \cite{kormann2015semi,ehrlacher2017dynamical,guo2022low}, the Boltzmann-type equation \cite{chikitkin2021numerical,boelens2018parallel,boelens2020tensor}, and the Fokker-Planck equation \cite{dolgov2012fast,chertkov2021solution}. In \cite{chikitkin2021numerical}, the authors solve the BTE with the Shakhov model collision operator by combining the discrete velocity method (DVM) \cite{platkowski1988discrete, palczewski1997consistency} and the tensor method, where full tensors formed by the probability density function (PDF) on the velocity grid are approximated using the low-rank Tucker decomposition. In \cite{boelens2018parallel,boelens2020tensor}, a continuous CANDECOMP/PARAFAC (CP) format is employed to solve the Boltzmann-BGK equation. In this method, the six-dimensional Boltzmann-BGK equation is converted into a set of six one-dimensional problems by using the low-rank CP format with Fourier bases. After that, these one-dimensional problems are addressed by the alternating least-squares (ALS) algorithm and the discrete Fourier transform. The inherent low-rank structure in the solutions and the operators of the Boltzmann-BGK equation allow tensor methods to offer advantages in terms of reduced computational and memory costs compared to traditional methods. However, as shown in \cite{boelens2020tensor}, solving the Boltzmann-BGK equation using tensor methods while simultaneously maintaining low-rank representations and accuracy is challenging.

In this paper, we present a novel accuracy-preserved tensor-train (APTT) method for solving the Boltzmann-BGK equation, aiming to maintain low-rank representations and preserve accuracy simultaneously. The main contributions of this work are as follows.
In the APTT method, we apply the Crank-Nicolson Leap Frog (CNLF) scheme and the second-order upwind scheme to discretize the Boltzmann-BGK equation, resulting in a fully tensor algebraic system. Subsequently, the full tensors in the algebraic system are recompressed into low-rank TT format tensors \cite{oseledets2011tensor}. During the recompression process, we identify a low-rank TT format tensor with the lowest TT-rank that satisfies specified accuracy tolerances, which is a critical factor for APTT to achieve accuracy preservation.
The key to APTT is to construct the low-rank TT representation of the collision term,
where the higher-order tensor representing the collision term is decomposed into the products of some lower-order and low-rank TT format tensors.
Based on the low-rank TT representations of PDF and the collision term,
the fully tensor algebraic system is then approximated by a TT-based low-rank linear system, which is efficiently solved using a low-rank linear solver, namely TT-modified alternating least-squares (TT-MALS) or density matrix renormalization group (DMRG) \cite{holtz2012alternating}. The use of low TT-rank representation in the APTT method significantly reduces computational and memory costs compared to traditional methods, especially for high-dimensional problems. Moreover, we establish the error analysis of the APTT method, which demonstrates that it maintains the same convergence rate as that of the discretization scheme by carefully setting tolerances in the APTT method. Furthermore, the error analysis shows that the solution of the APTT method satisfies the conservation laws of mass, momentum, and energy within the prescribed accuracy tolerances. Several numerical experiments validate the convergence rate and efficiency of the APTT method.

The remainder of this paper is organized as follows. In the next section, some related works are introduced. In \Cref{prob_set}, we provide a brief introduction to the Boltzmann-BGK equation. Then, in \Cref{sec:discretization}, we discretize the Boltzmann-BGK equation, yielding a fully tensor algebraic system. The APTT method is proposed in \Cref{sec:approx}. \Cref{analysis} presents complexity and convergence analysis of the APTT method. Numerical experiments are reported in \Cref{sec_res}, and the paper is concluded in \Cref{sec_con}.

\section{Related work}
Over the past few decades, various traditional methods for numerically solving Boltzmann-type equations have emerged, i.e., the direct simulation Monte Carlo (DSMC) method \cite{bird1970direct, bird1994molecular}, the particle-based method \cite{rjasanow1996stochastic, rjasanow1998reduction}, the discrete-velocity model (DVM) \cite{platkowski1988discrete, palczewski1997consistency}, and the spectral method \cite{wu2014solving, pareschi2000stability, wang2019approximation}. Stochastic methods, such as the DSMC and the particle-based method ensure efficiency and preservation of main physical properties but struggle with statistical fluctuations, especially in the presence of non-stationary flows or near continuum regimes. The DVM, a deterministic approach based on a Cartesian grid in velocity and a discrete collision mechanism, preserves main physical properties but tends to be less efficient compared to the DSMC method, and its accuracy is often limited to less than first order \cite{palczewski1997consistency, panferov2002new}. Spectral methods, including the fast Fourier spectral method \cite{pareschi2000stability, mouhot2006fast, wu2014solving} and the spectral method based on global orthogonal polynomials \cite{cai2015approximation, gamba2018galerkin, wang2019approximation}, approximate the PDF through spectral basis functions in the phase space and efficiently calculate collision terms with spectral accuracy. Despite their advantages in the preservation of conservation laws, these traditional methods require high computational and memory costs and exhibit low efficiency due to the curse of dimensionality.

The low efficiency of classical numerical methods in solving the Boltzmann-type equation is partially caused by grid discretization.
One way to tackle this problem is to introduce neural network approximation techniques. Instead of generating the mesh, random samples can be generated to train neural networks to approximate the solution of the Boltzmann-type equation, which avoids the high computational cost originating from grid discretization. Two main types of neural network methods have been employed to solve the Boltzmann-type equation in recent years. The first type, as demonstrated in studies such as \cite{han2019uniformly, huang2022machine, li2023learning, miller2022neural, schotthofer2022structure, xiao2021using, xiao2023relaxnet}, focuses on learning a closed reduced model for the Boltzmann-type equation. This approach concentrates on approximating the collision term and preserving various physical invariances. The second type, represented in works like \cite{de2021physics, li2022physics, li2024solving, lou2021physics}, directly solves the Boltzmann-type equation within the framework of physics-informed neural networks (PINN). By leveraging the advantages of neural networks in handling high-dimensional problems, neural network methods based on PINN exhibit promising potential in solving the Boltzmann-type equation.
However, the efficiency of training neural networks can be significantly affected by the network structure and hyperparameters \cite{Goodfellow-et-al-2016}. Introducing the low-rank structure or sparse representation of the PDF in the Boltzmann-type equation can expedite the training process of neural networks \cite{li2024solving}. Although the neural network methods can provide a new direction to solve the Boltzmann-type equation, developing provable convergence for such methods is challenging, and the accuracy of neural network methods cannot be guaranteed \cite{karniadakis2021physics, hao2022physics}.

\section{Boltzmann-BGK equation}\label{prob_set}
In the absence of external forces,
we consider the following Boltzmann equation \cite{cercignani1988boltzmann} describing rarefied gas dynamics:
\begin{equation}\label{original-eq_boltzmann}
	\frac{\partial f}{\partial t} + \bm{v} \cdot \nabla_{\bm{x}} f = Q(f,f),
\end{equation}
where $t \in (0, {t}^\star]$ is the time coordinate, $\bm{x} \in  \mathbb{R}^D$ ($D = 1,2,3$) denotes the spatial coordinates, $\bm{v} \in  \mathbb{R}^D$ represents the velocity, and $f:= f(\bm{x}, \bm{v}, t)$ is a PDF that estimates the number of particles with velocity $\bm{v}$ at position $\bm{x}$ and time $t$. 
$Q(f,f)$ is the collision operator describing the effects of internal forces due to particle interactions. In classical rarefied gas flows, the collision operator is defined as:
\begin{equation}
	\label{integal-operator}
	Q(f,f)=\int\limits_{\mathbb{R}^D}\int\limits_{\mathbb{S}^{D-1}}\mathbb{B}(\bm{v},\bm{v}_1,\mathbf{n})|f(\bm{x},\bm{v}',t)f(\bm{x},\bm{v}_1',t)-f(\bm{x},\bm{v},t)f(\bm{x},\bm{v}_1,t)|\hbox{d}\mathbf{n}\hbox{d} \bm{v}_1,
\end{equation}
where $\bm{v}$ and $\bm{v}_1$ are the velocities of two particles before the collision. $\mathbf{n}$ is the unit normal vector to the $(D-1)$-dimensional unit sphere $\mathbb{S}^{D-1}$. $\bm{v}'=(\bm{v}+\bm{v}_1+\|\bm{v}-\bm{v}_1\|_2\mathbf{n})/2$ and $\bm{v}'_1=(\bm{v}+\bm{v}_1-\|\bm{v}-\bm{v}_1\|_2\mathbf{n})/2$ are the velocities after the collision. The collision kernel $ \mathbb{B}(\bm{v},\bm{v}_1,\mathbf{n})$ is a non-negative function of the Euclidean 2-norm $\|\bm{v}-\bm{v}_1\|_2$ and the scattering angle between the relative velocities before and after the collision \cite{cercignani1997many}. The collision operator defined by \eqref{integal-operator} satisfies the conservation laws of mass, momentum, and energy  \cite{dimarco2014numerical}:
\begin{equation}
	\label{conservation-law}
	\int\limits_{\mathbb{R}^D}Q(f,f)\varphi(\bm{v})\hbox{d}\bm{v}=0,\qquad \varphi(\bm{v})=1 \hbox{, } \bm{v},\hbox{ or } \|\bm{v}\|_2^2,
\end{equation}
and the Boltzmann $H$-theorem:
\begin{equation}
	\label{H-theorem}
	\int\limits_{\mathbb{R}^D}Q(f,f)\log(f)\hbox{d}\bm{v}\leq0.
\end{equation}
According to the conservation laws and the Boltzmann $H$-theorem, any equilibrium PDF satisfying $Q(f,f)=0$ is locally Maxwellian:
\begin{equation}\label{eq:equilibrium-1}
	f_{\mathrm{eq}} (\bm{x}, \bm{v}, t) = \frac{\rho(\bm{x},t)}{(2\pi k_{\mathrm{B}} T(\bm{x},t)/{M})^{D/2}}\exp\left(-\frac{M\|\bm{v}-\bm{U}(\bm{x},t)\|_2^2}{2k_{\mathrm{B}}T(\bm{x},t)}\right),
\end{equation}
where $k_{\mathrm{B}}$ is the Boltzmann constant, $M$ is the particle mass. In \eqref{eq:equilibrium-1}, macroscopic variables $\rho,\,\bm{U},\,T$ denote the number density, mean velocity, and temperature of a gas, which are respectively defined as follows
\begin{equation}
	\label{macro-varibales-1}
	\begin{aligned}
		\rho(\bm{x},t)&=\int_{\mathbb{R}^D}f(\bm{x},\bm{v},t)\mathrm{d}\bm{v},\\
		\bm{U}(\bm{x},t)&:=(U_1(\bm{x},t),\cdots,U_D(\bm{x},t))=\frac{1}{\rho(\bm{x},t)}\int_{\mathbb{R}^D}\bm{v}f(\bm{x},\bm{v},t)\mathrm{d}\bm{v}, \\
		T(\bm{x},t)&=\frac{M}{D k_{\mathrm{B}}\rho(\bm{x},t)}\int_{\mathbb{R}^D}\|\bm{v}-\bm{U}(\bm{x},t)\|^2_2 f(\bm{x},\bm{v},t)\mathrm{d}\bm{v}.
	\end{aligned}
\end{equation}

The collision operator \eqref{integal-operator} is a nonlinear integral operator, posing challenges for designing efficient numerical algorithms. In this work, we consider the simplest operator satisfying \eqref{conservation-law} and \eqref{H-theorem}, which is the BGK operator \cite{bhatnagar1954model} given by
\begin{equation}\label{eq:collision}
	Q(f,f) = {\nu (\bm{x}, t)} \left[ f_{\mathrm{eq}} (\bm{x}, \bm{v}, t) - f(\bm{x}, \bm{v}, t)\right].
\end{equation}
Here the collision frequency $\nu (\bm{x}, t) > 0$ is usually defined as
$$ \nu(\bm{x}, t) = K\rho(\bm{x}, t)T(\bm{x}, t)^{1-\mu},$$
where $K>0$ is a constant and the exponent $\mu$ of the viscosity law depends on the molecular interaction potential and the type of the gas. With the simplification of the collision operator, the Boltzmann equation \eqref{original-eq_boltzmann} yields the Boltzmann-BGK equation
\begin{equation}\label{eq_boltzmann-1}
	\frac{\partial f}{\partial t} + \bm{v} \cdot \nabla_{\bm{x}} f = {\nu (\bm{x}, t)} \left[ f_{\mathrm{eq}} (\bm{x}, \bm{v}, t) - f(\bm{x}, \bm{v}, t)\right].
\end{equation}

For convenience of description, let us introduce several dimensionless parameters. Assume that the Boltzmann-BGK \eqref{eq_boltzmann-1} is posed on a $2D$-dimensional hypercube, i.e. $f:\Omega_{\bm{x}}\times \Omega_{\bm{v}}\times (0,t^\star]\rightarrow \mathbb{R}^+$ with $\Omega_{\bm{x}}=[-b_{\bm{x}},b_{\bm{x}}]^D$ and $\Omega_{\bm{v}}=[-b_{\bm{v}},b_{\bm{v}}]^D$ representing the spatial domain
and the velocity domain, respectively. Let us introduce the following dimensionless mapping
\begin{equation}
	\begin{aligned}
		&\hat{\bm{x}}=\frac{\bm{x}\pi}{b_{\bm{x}}},\quad \hat{\bm{v}}=\frac{\bm{v}\pi}{b_{\bm{v}}},\quad \hat{t}=\frac{tb_{\bm{v}}}{b_{\bm{x}}},
		\quad \\
		&\hat{\bm{U}}(\hat{\bm{x}},\hat t)=\frac{\bm{U}({\bm{x}}, t)\pi}{b_{\bm{v}}},\quad  \hat{T}(\hat{\bm{x}},\hat t)=\frac{T({\bm{x}}, t)}{T_c},\quad \hat{\nu}(\hat{\bm{x}},\hat t)=\frac{\nu({\bm{x}}, t)\lambda}{b_{\bm{v}}},
	\end{aligned}
\end{equation}
where $\lambda$ is the mean free path of a gas molecule, and $T_c$ is a characteristic temperature. By rescaling the PDF $\hat{f}(\hat{\bm{x}},\hat{\bm{v}},\hat t)={f}({\bm{x}},{\bm{v}}, t)b_{\bm{x}}^Db_{\bm{v}}^D/\pi^{2D}$, the Boltzmann-BGK \eqref{eq_boltzmann-1} can be rewritten as
\begin{equation}\label{eq_boltzmann}
	\frac{\partial \hat f}{\partial t} + \hat{\bm{v}} \cdot \nabla_{\hat{\bm{x}}} \hat f =\frac {\hat{\nu} (\hat{\bm{x}}, \hat t)} {\mathrm{Kn}}\left[ \hat {f}_{\mathrm{eq}} (\hat{\bm{x}}, \hat{\bm{v}},\hat t) - \hat f(\hat{\bm{x}}, \hat{\bm{v}},\hat t)\right],
\end{equation}
where equilibrium PDF is
\begin{equation}\label{eq:equilibrium}
	\hat {f}_{\mathrm{eq}} (\hat{\bm{x}}, \hat{\bm{v}},\hat t) = \frac{\hat{\rho}(\hat{\bm{x}},\hat {t})}{(2\pi \hat {T}(\hat{\bm{x}},\hat t)/\mathrm{Bo})^{D/2}}\exp\left(-\mathrm{Bo}\frac{\|\hat{\bm{v}}-\hat{\bm{U}}(\hat{\bm{x}},\hat t)\|_2^2}{2\hat T(\hat{\bm{x}},\hat t)}\right)
\end{equation}
with the dimensionless Knudsen (Kn) and Boltzmann (Bo) numbers being $\mathrm{Kn}=\frac{\lambda}{b_{\bm{x}}}$ and $\mathrm{Bo}=\frac{M b_{\bm{v}}^2}{\pi^2k_{\mathrm B} T_c}$, respectively.
The number density, mean velocity, and temperature are defined as follows
\begin{equation}
	\label{macro-varibales}
	\begin{aligned}
		\hat{\rho}(\hat{\bm{x}},\hat t)&=\int_{[-\pi,\pi]^D}\hat f(\hat{\bm{x}}, \hat{\bm{v}},\hat t)\mathrm{d}\hat{\bm{v}},\\
		\hat{\bm{U}}(\hat{\bm{x}},\hat t)&=\frac{1}{\hat{\rho}(\hat{\bm{x}},\hat t)}\int_{[-\pi,\pi]^D}\hat{\bm{v}}\hat f(\hat{\bm{x}}, \hat{\bm{v}},\hat t)\mathrm{d}\hat{\bm{v}}, \\
		\hat T(\hat{\bm{x}},\hat t)&=\frac{\mathrm{Bo}}{D \hat{\rho}(\hat{\bm{x}},\hat t)}\int_{[-\pi,\pi]^D}\|\hat{\bm{v}}-\hat{\bm{U}}(\hat{\bm{x}},\hat t)\|^2_2\hat f(\hat{\bm{x}}, \hat{\bm{v}},\hat t)\mathrm{d}\hat{\bm{v}}.
	\end{aligned}
\end{equation}
In the following of this paper, we drop the hat below and continue to use the dimensionless
quantities.

The initial condition for the Boltzmann-BGK equation \eqref{eq_boltzmann} is set as $f(\bm{x}, \bm{v}, t):=f_0(\bm{x}, \bm{v})$.
According to \eqref{macro-varibales}, the locally Maxwellian $f_{\mathrm{eq}} (\bm{x}, \bm{v}, t)$ in \eqref{eq:equilibrium} and the collision frequency $\nu(\bm{x}, t) $ are nonlinear functionals of the PDF $f (\bm{x}, \bm{v}, t)$. Therefore, the Boltzmann-BGK equation \eqref{eq_boltzmann} is a nonlinear PDE in $2D+1$ dimension. Numerically solving the Boltzmann-BGK is computationally expensive, especially for fine meshes. To address this issue, we will propose an efficient algorithm based on tensor-train decomposition to numerically solve the Boltzmann-BGK equation in the following of this paper.

\section{Discretization of Boltzmann-BGK}\label{sec:discretization}
In this section, we introduce a second-order finite difference scheme in both temporal and spatial directions to discretize the Boltzmann-BGK equation. The finite difference scheme is described in tensor form.
Let us divide the time interval $(0,t^\star]$ into several subintervals $(t_n,\,t_{n+1}]$ with fixed time step size $\Delta t$ and $t_n=n\Delta t$. Here, $n=0,\,1,\,\cdots,\,\frac{t^\star}{\Delta t}-1$. For ease of description, we assume that $D=3$. The periodic boundary condition is applied in $\bm{x}$ direction and the homogeneous Dirichlet boundary condition is used in $\bm{v}$ direction. The finite difference scheme for the Boltzmann-BGK equation with $D=1,\,2$ can be obtained in a similar manner. The computational domain $[-\pi,\,\pi]^{6}$ is covered by a uniform tensor product grid with mesh size $h=2\pi/m$. Let us denote $(\bm{x}^{\bm{k}},\,\bm{v}^{\bm{l}}):=(h\bm{k},h\bm{l})-(\pi,\,\pi,\,\cdots,\,\pi)$ with $\bm{k}=(k_1,\,k_2,\,k_3)$, $\bm{l}=(l_1,\,l_2,\,l_3)$, and ${k}_i,\,{l}_i=0,\,1,\,\cdots,\,m-1$, respectively. The solution of the Boltzmann-BGK equation at $t=t_n$ is approximated as a sixth-order tensor $\tensor{F}^n:=\left(\tensor{F}^n_{k_1,k_2,k_3,l_1,l_2,l_3}\right)\in\mathbb{R}^{m\times m\cdots\times m}$ with $\tensor{F}^n_{k_1,k_2,k_3,l_1,l_2,l_3}:=\tensor{F}^n_{\bm{k},\bm{l}}\approx f(\bm{x}^{\bm{k}},\bm{v}^{\bm{l}},t_n)$.  Assuming the value of tensor $\tensor{F}^n$ is known, we can calculate the number density, mean velocity, and temperature using a suitable numerical integration method, resulting in third-order tensors $\tensor{\rho}^n:=\left(\tensor{\rho}^n_{\bm{k}}\right)$, $\tensor{U}^n:=\left(\tensor{U}^n_{\bm{k}}\right)$, and $\tensor{T}^n:=\left(\tensor{T}^n_{\bm{k}}\right)$, respectively. To match the accuracy of the second-order discretization scheme, we use the trapezoid rule, a second-order numerical integration method, to calculate these integrals.
Specifically, for the number density $\rho(\bm{x})$, it follows
\begin{equation}\label{eq:integral-rho}
	\begin{split}
		\rho(\bm{x}^{\bm{k}},t_n) = \int_{[-\pi,\pi]^3}f(\bm{x}^{\bm{k}},\bm{v},t_n)\mathrm{d}\bm{v}
		\approx h^3\sum\limits_{\bm{l}} \tensor{F}^n_{\bm{k},\bm{l}}:=\tensor{\rho}^n_{\bm{k}}.
	\end{split}
\end{equation}
{Similarly, for the mean velocity $\bm{U}(\bm{x})$ and the temperature $T(\bm{x})$, we have}
\begin{equation}\label{eq:integral-U}
	\begin{split}
		\bm{U}(\bm{x}^{\bm{k}},t_n) = \frac1{\rho(\bm{x}^{\bm{k}},t_n)}{\int_{[-\pi,\pi]^3}\bm{v}f(\bm{x}^{\bm{k}},\bm{v},t_n)\mathrm{d}\bm{v}}
		\approx \frac {h^3} {\tensor{\rho}^n_{\bm{k}}} \sum\limits_{\bm{l}}\bm{v}^{\bm{l}} \tensor{F}^n_{\bm{k},\bm{l}}
		:=\tensor{U}^n_{\bm{k}},
	\end{split}
\end{equation}
and
\begin{equation}\label{eq:integral-T}
	\begin{split}
		T(\bm{x}^{\bm{k}},t_n) = &\frac{\mathrm{Bo}}{3 {\rho}({\bm{x}}^{\bm{k}}, t)}\int_{[-\pi,\pi]^3}\|\bm{v}-\bm{U}(\bm{x}^{\bm{k}},t_n)\|_2^2f(\bm{x},\bm{v},t_n)\mathrm{d}\bm{v}\\
		\approx&\frac{\mathrm{Bo}h^3}{3 {\tensor{\rho}}^n_{\bm{k}}}\sum\limits_{\bm{l}}\|\bm{v}^{\bm{l}}-\tensor{U}^n_{\bm{k}}\|_2^2 \tensor{F}^n_{\bm{k},\bm{l}}
		:=\tensor{T}^n_{\bm{k}}.
	\end{split}
\end{equation}
Substituting them into \eqref{eq:equilibrium} and \eqref{eq:collision}, it yields a sixth-order tensor $\tensor{Q}^n:=\left(\tensor{Q}^n_{\bm{k},\bm{l}}\right)$ $\in\mathbb{R}^{m\times m\cdots\times m}$, whose element approximates the collision operator $Q$ at $(\bm{x}^{\bm{k}},\bm{v}^{\bm{l}},t_n)$.

We then introduce a second-order upwinding scheme to discretize $\bm{v}\cdot\nabla_{\bm{x}}f$ in \eqref{eq_boltzmann}. By setting $\bm{e}_1=(1,0,0)$, $\bm{e}_2=(0,1,0)$, $\bm{e}_3=(1,0,0)$, and $x_i=\bm{x}\cdot\bm{e}_i$, we define two useful notations as following
\begin{equation}
	\begin{aligned}
		D_{x_i}^+\tensor{F}^n_{\bm{k},\bm{l}}&=\frac{3\tensor{F}^n_{\bm{k},\bm{l}}-4\tensor{F}^n_{\bm{k}-\bm{e}_i,\bm{l}}+\tensor{F}^n_{\bm{k}-2\bm{e}_i,\bm{l}}}{2h},\\ D_{x_i}^-\tensor{F}^n_{\bm{k},\bm{l}}&=-\frac{3\tensor{F}^n_{\bm{k},\bm{l}}-4\tensor{F}^n_{\bm{k}+\bm{e}_i,\bm{l}}+\tensor{F}^n_{\bm{k}+2\bm{e}_i,\bm{l}}}{2h}.
	\end{aligned}
\end{equation}
Then, the second-order upwinding scheme is defined as follows:
\begin{equation}\label{eq:upwind}
	\bm{v}^{\bm{l}}\cdot\nabla_{\bm{x}}f(\bm{x}^{\bm{k}},\bm{v}^{\bm{l}},t_n) \approx [\bm{v}^{\bm{l}}\cdot\nabla_{\bm{x}}\tensor{F}^n_{\bm{k},\bm{l}}]_u:=\sum\limits_{i=1}^3 \left((v_i^{l_i})^{+}D_{x_i}^+\tensor{F}^n_{\bm{k},\bm{l}}+(v_i^{l_i})^{-}D_{x_i}^-\tensor{F}^n_{\bm{k},\bm{l}}\right),
\end{equation}
where $v_i^{l_i}=\bm{v}^{\bm{l}}\cdot\bm{e}_i$,  $(v_i^{l_i})^{+} = \max(v_i^{l_i},0)$, and $(v_i^{l_i})^-=\min(v_i^{l_i},0)$, respectively. The time derivative $\frac{\partial f}{\partial t}$ at $t=t_n$ is discretized by using the CNLF scheme \cite{boelens2020tensor}, which is a second-order semi-implicit method and widely used in many applications such as atmosphere and ocean \cite{asselin1972frequency,thomas2005ncar,williams2011raw,kubacki2013uncoupling}. The CNLF scheme is stable and semi-implicit, allowing for less restriction in the choice of the time step size \cite{layton2012stability,kubacki2013uncoupling,jiang2015crank}. More importantly, as compared with a fully implicit scheme, it avoids solving the nonlinear system caused by the collision operator, providing the possibility for dimensionality reduction methods based on low-rank tensor representations. The fully discrete system for the Boltzmann-BGK equation \eqref{eq_boltzmann} is given as
\begin{equation}\label{eq:cnlf}
	\frac{\tensor{F}^{n+1}_{\bm{k},\bm{l}}-\tensor{F}^{n-1}_{\bm{k},\bm{l}}}{2\Delta t} =- \frac{[\bm{v}^{\bm{l}}\cdot\nabla_{\bm{x}}\tensor{F}^{n+1}_{\bm{k},\bm{l}}]_u+[\bm{v}^{\bm{l}}\cdot\nabla_{\bm{x}}\tensor{F}^{n-1}_{\bm{k},\bm{l}}]_u}{2} + \tensor{Q}^n_{\bm{k},\bm{l}},\quad \hbox{for all }\bm{k},\bm{l}.
\end{equation}

Let us introduce two matrices $\bm{D}^+$, $\bm{D}^-\in\mathbb{R}^{m\times m}$ defined as
\begin{equation}\label{eq:upwind-mat}
	\bm{D}^{+} = \frac{1}{2h}\left[\begin{array}{ccccc}
		3 &  & &1 &-4  \\
		-4 &3 & & & 1 \\
		1&-4 &3 & &  \\
		& \ddots&\ddots &\ddots &  \\
		& & 1& -4&3  \\
	\end{array}\right]\ \text{and}\ \bm{D}^{-} = -\left(\bm{D}^{+}\right)^T.
\end{equation}
Two diagonal matrices $\bm{V}^{+},\,\bm{V}^{-}\in\mathbb{R}^{m\times m}$ are respectively given by
\begin{equation}
	\label{eq:upwind-velocity}
	\begin{aligned}
		\bm{V}^{+}&=\hbox{diag}\left\{(v_1^0)^+,(v_1^1)^+,\cdots,(v_1^{m-1})^+\right\}, \\
		\bm{V}^{-}&=\hbox{diag}\left\{(v_1^0)^-,(v_1^1)^-,\cdots,(v_1^{m-1})^-\right\}.
	\end{aligned}
\end{equation}
{Let
	$\tensor{I},\tensor{L}\in\mathbb{R}^{m^6\times m^6}$ be two matrices corresponding to identity matrix and the second-order upwind scheme \eqref{eq:upwind}, respectively. The definitions of the two matrices are respectively given as
	\begin{equation*}
		\tensor{I} = \bm{I}\otimes\bm{I}\otimes\bm{I}\otimes\bm{I}\otimes\bm{I}\otimes\bm{I},
	\end{equation*}
	and
	\begin{equation}\label{eq:coefficient}
		\begin{split}
			\tensor{L} = &-\bm{V}^+\otimes \bm{I}\otimes \bm{I}\otimes \bm{D}^+ \otimes \bm{I}\otimes \bm{I} - \bm{V}^-\otimes \bm{I}\otimes \bm{I}\otimes \bm{D}^- \otimes \bm{I}\otimes \bm{I}\\
			&-\bm{I}\otimes\bm{V}^+\otimes  \bm{I}\otimes  \bm{I}\otimes \bm{D}^+ \otimes  \bm{I} - \bm{I}\otimes \bm{V}^-\otimes \bm{I}\otimes \bm{I}\otimes \bm{D}^- \otimes \bm{I}\\
			&-\bm{I}\otimes  \bm{I}\otimes\bm{V}^+\otimes  \bm{I}\otimes  \bm{I}\otimes \bm{D}^+  - \bm{I} \otimes  \bm{I}\otimes \bm{V}^-\otimes \bm{I}\otimes \bm{I}\otimes \bm{D}^-,
		\end{split}
	\end{equation}
	where $\bm{I}$ is the $m\times m$ identity matrix and $\otimes$ denotes the tensor product. According to \cite{kazeev2012low, oseledets2011tensor,lee2018fundamental}, we can rewrite the two matrices $\tensor{I},\tensor{L}$ as low-rank TT format operators, with rank $r= 1$ and 6, respectively. The low-rank TT format operators are still denoted as $\tensor{I}$ and $\tensor{L}$, respectively. Then we can rewrite the fully discrete system \eqref{eq:cnlf} in tensor form as follows:
	\begin{equation}\label{eq:linear1}
		(\tensor{I}-\Delta t\tensor{L})\tensor{F}^{n+1} = (\tensor{I}+\Delta t\tensor{L})\tensor{F}^{n-1}+2\Delta t\tensor{Q}^{n}.
\end{equation}}

The second-order upwinding scheme used in fully discrete system \eqref{eq:linear1} has good stable performance in numerical simulations. However, due to the hyperbolic property of the Boltzmann-BGK equation, the fully discrete system \eqref{eq:linear1} with relatively large time step size will suffer numerical oscillations. To suppress it, following the idea of \cite{strikwerda2004finite}, we add an artificial dissipation into the upwind scheme \eqref{eq:upwind}. Let matrix $\bm{D}\in\mathbb{R}^{m\times m}$ be
\begin{equation}
	\bm{D} = \frac{1}{h^4}\left[\begin{array}{cccccc}
		6 & -4 & 1 & &1 &-4  \\
		-4 &6 & -4 & 1 & & 1 \\
		1&-4 &6 &-4 &1 &  \\
		& \ddots&\ddots&\ddots &\ddots &\ddots  \\
		1 &  & 1 &-4& 6&-4  \\
		-4 &1 &  &1 & -4&6  \\
	\end{array}\right]\
\end{equation}
corresponding to a fourth-order central finite difference scheme for one-dimensional operator $\Delta^2$. By setting $\overline{\tensor{M}}\in\mathbb{R}^{m^6\times m^6}$ be a matrix defined as
\begin{equation}
	\overline{\tensor{M}} :=\bm{I}\otimes\bm{I}\otimes  \bm{I}\otimes  \left(\bm{D} \otimes\bm{I}\otimes   \bm{I} + \bm{I}\otimes \bm{D} \otimes  \bm{I} +\bm{I} \otimes  \bm{I} \otimes \bm{D}\right),
\end{equation}
the fully discrete system for the Boltzmann-BGK equation \eqref{eq_boltzmann} with artificial dissipation is given as following:
\begin{equation}\label{eq:modified_upwind}
	(\tensor{I}-\Delta t\tensor{L})\tensor{F}^{n+1} = (\tensor{I}+\Delta t\tensor{L})\tensor{F}^{n-1}+2\Delta t \tensor{Q}^n-\frac{\epsilon h^4}{2^4}\tensor{M}\tensor{F}^{n-1},
\end{equation}
where $\epsilon$ is the coefficient of the artificial viscosity term and $\tensor M$ is a low-rank TT format operator with rank $r=3$ corresponding to matrix $\overline{\tensor{M}}$. As reported in \Cref{sec_discontinuous_init}, the artificial dissipation term is important and indispensable in solving the Boltzmann-BGK equation with discontinuous initial distributions.

\begin{rmk}
	We remark that other difference schemes on uniform meshes such as the five-point centered-difference formula can be used for spatial discretization since the coefficient matrix of its corresponding linear system also enjoys the low-rank representation similar to \eqref{eq:coefficient}.
\end{rmk}

\section{Accuracy-preserved tensor-train method}\label{sec:approx}
Using classical iterative methods, such as generalized minimal residual (GMRES) \cite{saad1986gmres} and biconjugate gradient stabilized \cite{van1992bi} to solve the linear system \eqref{eq:linear1} or \eqref{eq:modified_upwind}, results in total computational complexity and memory cost no less than $\mathcal{O}(m^{2D})$. Consequently, solving the Boltzmann-BGK equation with $D=2$ or $3$ becomes prohibitively expensive as the number of grids increases.  To address this issue, we develop a novel low-rank solver to solve the linear system \eqref{eq:linear1} or \eqref{eq:modified_upwind} based on the TT format \cite{oseledets2011tensor}, reducing total computational complexity and memory cost from $\mathcal{O}(m^{2D})$ to $\mathcal{O}(m^{D+1}r)$, where $r$ is the rank of $\tensor{F}^n$. Since the numerical error between the exact solution and the low-rank solution obtained by the low-rank solver is proved to be bounded by several prescribed accuracy tolerances, we call the low-rank solver as the APTT solver.
To this end, we first introduce a low-rank TT format tensor to approximate the full tensor for the PDF and then construct the right-hand side of \eqref{eq:linear1} or \eqref{eq:modified_upwind} based on the low-rank TT format tensor. Finally, the APTT method is summarized in \Cref{sec:tt-solver}.

\subsection{Low-rank tensor-train format tensor and recompression}
\label{sec:recompression}
Let us use bold calligraphic letters to represent full tensors or TT format tensors.
A low-rank $d$th-order  tensor $\tensor{A}\in\mathbb{R}^{I_1\times I_2 \cdots\times I_d}$ with TT format can be defined as follows  \cite{oseledets2011tensor}
\begin{equation}\label{eq:tt-format}
	\tensor{A}_{i_1,i_2,\cdots,i_d}:=\sum\limits_{\alpha_1,\alpha_2,\cdots,\alpha_{d-1}}\tensor{A}_{i_1,\alpha_1}^{1}\tensor{A}^{2}_{\alpha_1,i_2,\alpha_2}\cdots\tensor{A}_{\alpha_{d-1},i_d}^{d},
\end{equation}
where $\tensor{A}^{1}\in\mathbb{R}^{I_1\times r_1}$, $\tensor{A}^{d}\in\mathbb{R}^{r_{d-1}\times I_d}$, and $\tensor{A}^{i}\in\mathbb{R}^{r_{i-1}\times I_i\times r_i}$ are TT cores, and $(r_1,r_2,\cdots,r_{d-1})$ is called the TT-rank of $\tensor{A}$. For convenience, the TT format tensor $\tensor{A}$ is generally denoted as $\mathbb{TT}(\tensor{A}^{1},\tensor{A}^{2},\cdots,\tensor{A}^{d})$.
With the low-rank TT representation, the memory cost of the $d$th-order tensor $\tensor{A}$ will be reduced from $\mathcal{O}(I_1I_2\cdots I_d)$ to $\mathcal{O}(I_1r_1+r_1I_2r_2\cdots+I_dr_{d-1})$, which only grows linearly with $d$. Based on the low-rank TT format, the newly proposed APTT algorithm aims to find a low-rank TT approximation for tensor $\tensor{F}^{n+1}$ by solving the tensor system \eqref{eq:linear1} or \eqref{eq:modified_upwind}. At each time step, the system \eqref{eq:linear1} or \eqref{eq:modified_upwind} is solved by implementing several basic tensor operations, such as addition and the Hadamard product. However, this often results in the rapid increase of the TT-ranks for the solution $\tensor{F}^{n+1}$ and growing computational complexity \cite{oseledets2011tensor, boelens2018parallel,boelens2020tensor}. To suppress the excessive growth of TT-rank, we introduce the recompression algorithm for rank reduction.

For the $d$th-order full tensor or a TT format tensor $\tensor{A}$ with a relatively large rank, there are two types of recompression algorithms. The first type of algorithms \cite{oseledets2011tensor,dolgov2013tt} is to predetermine a target truncation rank $\bm{r}$, and then find the best rank-$\bm{r}$ approximation of $\tensor{A}$ by solving
\begin{equation*}
	\min\limits_{\text{rank}(\tensor{B})\leq\bm{r}} \|\tensor{A}-\tensor{B}\|_F.
\end{equation*}
These algorithms can easily control the rank of TT format tensor, but may result in a loss of accuracy
when the predetermined truncation rank is inappropriate \cite{oseledets2011tensor,ehrlacher2022sott}. An alternative strategy for rank reduction of TT format tensors aims to find a tensor $\tensor{B}$ with the lowest TT-rank such that
\[
\|\tensor{A}-\tensor{B}\|_F\leq\varepsilon_b\|\tensor{A}\|_F,
\]
where $\varepsilon_b\in(0,1)$ is a prescribed accuracy tolerance. This type of algorithms is essentially equivalent to finding the best $\varepsilon_b$-approximation for a tensor $\tensor{A}$ in the TT format, which can be solved by high-performance TT-rounding algorithms \cite{oseledets2011tensor,al2022parallel,daas2022parallel}, such as the classical TT-SVD algorithm. Since we lack prior information about the ranks of tensors in the algebraic system \eqref{eq:linear1} and aim to control the accuracy of the low-rank approximation, we select the best $\varepsilon_b$-approximation for recompression in the APTT method.
For more details about basic operations on the TT format, we refer the readers to \cite{oseledets2011tensor,lee2018fundamental,cichocki2016tensor}.

\subsection{Low-rank TT format tensor of the right-hand side}
In this part, we will outline the procedure for constructing the low-rank TT format tensors for the full tensors in the right-hand side of \eqref{eq:linear1} and \eqref{eq:modified_upwind}, which includes terms:
$(\tensor{I}+\Delta t\tensor{L})\tensor{F}^{n-1}$, $2\Delta t\tensor{Q}^n$, and $-\frac{\epsilon h^4}{2^4}\tensor{M}\tensor{F}^{n-1}$.
The operators $\tensor{I}$, $\tensor{L}$, and $\tensor{M}$ are known to be low-rank TT tensor operators with TT-ranks at most $1$, $6$, and $3$, respectively.
Initially, $\tensor{F}^0$ is a $2D$th-order full tensor. We use the TT-SVD algorithm to find the best $\epsilon_b$-approximation for it and still denote it as $\tensor{F}^0$.
Assume that $\tensor{F}^{n-1}$ or $\tensor{F}^{n}$ is a low-rank TT format tensor with relatively small TT-rank, which is the initial low TT-rank tensor $\tensor{F}^0$ or a low TT-rank solution at $(n-1)$-th or $n$-th time step.
The two linear terms, $(\tensor{I}+\Delta t\tensor{L})\tensor{F}^{n-1}$ and $-\frac{\epsilon h^4}{2^4}\tensor{M}\tensor{F}^{n-1}$, are then approximately calculated by matrix-vector products in TT formats, which preserves the low TT-rank structure.

Next, let's discuss the construction of the low-rank TT format tensor for the collision operator $\tensor{Q}^n$. Due to the nonlinearity of the collision operator, constructing the low-rank TT format tensor $\tensor{Q}^n$ poses a challenge. There are primarily two types of methods to address this challenge. The first type of methods computes the TT format tensor $\tensor{Q}^n$ using the low-rank TT format tensor $\tensor{F}^n$.
Since the collision operator $Q$ is a nonlinear functional of $f$, the $2D$th-order tensor $\tensor{Q}^n$ is also a nonlinear operator on the $2D$th-order tensor $\tensor{F}^n$. Consequently, if $\tensor{F}^n$ is approximated by a low-rank TT representation, finding a low-rank TT representation for $\tensor{Q}^n$ becomes a challenging task. Nonlinear calculations typically result in a rapid increase in rank and corresponding computational complexity \cite{espig2020iterative}. Additionally,
most nonlinear operators are calculated using an iterative method, such as Newton's iteration. However, designing a nonlinear iterative method based on low-rank TT formats that can ensure both accuracy and a fast convergence rate remains challenging  \cite{cichocki2016tensor,lee2018fundamental,espig2020iterative}.
An alternative approach is to transform the low-rank TT format tensor $\tensor{F}^n$ into a full $2D$th-order tensor and then use it to compute the $2D$th-order tensor $\tensor{Q}^n$. Subsequently, the full $2D$th-order tensor $\tensor{Q}^n$ is converted back into a low-rank TT format tensor. However, this method incurs a computational cost on the order of $m^{2D}r^2$, which is inefficient, particularly for $D\geq 2$.

In order to balance the computational cost and accuracy, we introduce a new approach to calculate the low-rank TT format of $\tensor{Q}^n$ based on partial reduction of the TT format tensor $\tensor{F}^n$.
For $D=3$, let us assume that the low-rank TT format tensor $\tensor{F}^n$ is given by
\begin{equation}
	\tensor{F}^n_{k_1,k_2,k_3,l_1,l_2,l_3} = \sum\limits_{\alpha_1,\alpha_2,\cdots,\alpha_{5}}
	\tensor{G}_{k_1,\alpha_1}^{1,n}\tensor{G}^{2,n}_{\alpha_1,k_2,\alpha_2}\tensor{G}^{3,n}_{\alpha_2,k_3,\alpha_3}\tensor{G}^{4,n}_{\alpha_3,l_1,\alpha_4}\tensor{G}^{5,n}_{\alpha_4,l_2,\alpha_5}\tensor{G}_{\alpha_{5},l_3}^{6,n}.
\end{equation}
First, we compute an intermediate fourth-order full tensor $\tensor{Y}^n\in\mathbb{R}^{m\times m\times m\times r_3}$ by partially reducing the TT format tensor $\tensor{F}^n$ as follows:
\begin{equation}\label{eq:temtensor}   \tensor{Y}^n_{k_1,k_2,k_3,\alpha_3}=\sum\limits_{\alpha_1,\alpha_2}\tensor{G}^{1,n}_{k_1,\alpha_1}\tensor{G}^{2,n}_{\alpha_1,k_2,\alpha_2}\tensor{G}^{3,n}_{\alpha_2,k_3,\alpha_3}.
\end{equation}
The full tensor $\tensor{Y}^n$ is then used to construct third-order full tensors for the number density $\tensor{\rho}^n$, velocity $\tensor{U}^n$, and temperature $\tensor{T}^n$. Details of this step are summarized in \Cref{algo:rho-u-t}. Secondly, we use the third-order full tensors $\tensor{\rho}^n$,  $\tensor{U}^n$, and $\tensor{T}^n$ to calculate the tensor $\tensor{Q}^n$. Recalling \eqref{eq:collision}, due to the discretization of the equilibrium PDF $f_{\text{eq}}(\bm{x},\bm{v},t_n)$, the direct calculation of the equilibrium PDF results in a sixth-order full tensor and the corresponding sixth-order full tensor $\tensor{Q}^n$.

To avoid generating a sixth-order full tensor, we rewrite \eqref{eq:equilibrium} as
\begin{equation}\label{eq:equilibrium-2}
	\begin{split}
		f_{\mathrm{eq}} (\bm{x}, \bm{v}, t) = \prod\limits_{i=1}^{3}\frac{\rho(\bm{x},t)^{1/3}}{(2\pi  {T}({\bm{x}}, t)/\mathrm{Bo})^{1/2}}\exp\left(-\mathrm{Bo}\frac{({v}_i-{U}_i(\bm{x},t))^2}{2T(\bm{x},t)}\right):=\prod\limits_{i=1}^{3} f_{\mathrm{eq}}^i(\bm{x}, {v}_i, t).
	\end{split}
\end{equation}
Let sixth-order full tensor $\tensor{ F}^n_{\mathrm{eq}}$ and fourth-order full tensors $\tensor{F}^{i,n}_{\mathrm{eq}}$ with $i=1,2,3$, be the discretizations of $f_{\mathrm{eq}} (\bm{x}, \bm{v}, t)$ and $f_{\mathrm{eq}}^i(\bm{x}, {v}_i, t)$ at $t=t_n$, respectively. According to \eqref{eq:equilibrium-2}, we have
\begin{equation}
	\label{tensor-prodcut-structure}
	[\tensor{F}_{\text{eq}}^{n}]_{k_1,k_2,k_3,l_1,l_2,l_3} = \prod\limits_{i=1}^3[\tensor{F}_{\text{eq}}^{i,n}]_{k_1,k_2,k_3,l_i},
\end{equation}
which implies that the sixth-order full tensor $\tensor{F}^n_{\mathrm{eq}}$ can be represented as the product of three fourth-order full tensors $\tensor{F}^{i,n}_{\mathrm{eq}}$. This tensor product structure \eqref{tensor-prodcut-structure} remains valid even when low-rank TT format tensors are employed to approximate the four full tensors.
Therefore, we compute the fourth-order tensors $\tensor{F}^{i,n}_{\mathrm{eq}}$ with $i=1,2,3$ and subsequently convert the fourth-order tensor $\tensor{F}^{i,n}_{\mathrm{eq}}$ into a low-rank TT format tensor $\mathbb{TT}(\tensor{F}_{\text{eq}}^{i,n})$. The computational cost of this process is on the order of $m^{D+1}r^2$, which is two orders of $m$ less than the computational cost of sixth-order tensor compression. Then, we introduce an \texttt{expand} operator to convert fourth-order TT format tensor $\mathbb{TT}(\tensor{F}_{\text{eq}}^{i,n})$ into a sixth-order TT format tensor $\tensor{E}_{\text{eq}}^{i,n}$ (Step 3 in \Cref{algo:rhs}). For a $d$th-order tensor $\tensor{A} = \mathbb{TT}(\tensor{A}^{1},\ldots,\tensor{A}^{d})$ with the TT-rank $(r_1,\ldots,r_{d-1})$, the $\texttt{expand}$ operator can be defined as
\begin{equation}\label{eq:expand}
	\texttt{expand}(\tensor{A}, i, m) = \mathbb{TT}(\tensor{A}^{1}, \ldots, \tensor{A}^{i-1},\tensor{I}^{i}, \tensor{A}^{i}, \ldots,\tensor{A}^{d}),
\end{equation}
where $\tensor{I}^{i} \in \mathbb{R}^{r_i\times m\times r_i}$ with $\tensor{I}^{i}_{:,j,:}$ is a $r_i\times r_i$ identity matrix for all $j=0,1,\ldots,m-1$.  The low-rank TT format tensor for the equilibrium PDF is computed by $\tensor{F}^n_{\text{eq}}:=\tensor{E}^{1,n} \odot \tensor{E}^{2,n} \odot \tensor{E}^{3,n}$, where $\odot$ represents the Hadamard product. We summarize the construction of $\tensor{F}^n_{\text{eq}}$ in \Cref{algo:rhs}.
The discretization of the collision frequency $\nu (\bm{x}, t)$ can be calculated by using full tensors $\tensor{\rho}^n$ and $\tensor{T}^n$, resulting in a third-order full tensor. Similar to $\tensor{F}_{\text{eq}}^{i,n}$, we can construct a sixth-order low-rank TT format tensor $\tensor{\nu}^n$ for the collision frequency. Finally, the low-rank TT format tensor $\tensor{Q}^n$ is calculated by
\begin{equation}\label{eq:expand-1}
	\tensor{Q}^n=\frac{1}{\mathrm{Bo}}\tensor{\nu}^n\odot (\tensor{F}^n_{\text{eq}}-\tensor{F}^n).
\end{equation}

\begin{algorithm}
	\normalsize
	\caption{Computing third-order full tensors $\tensor{\rho}^n$, $\tensor{U}^n$, and $\tensor{T}^n$.}
	\label{algo:rho-u-t}
	\begin{algorithmic}[1]
		\Require TT cores of $\tensor{F}^n$: $\tensor{G}^{1,n}\in\mathbb{R}^{m\times r_1}$, $\tensor{G}^{6,n}\in\mathbb{R}^{r_5\times m}$, and $\{\tensor{G}^{i,n}\in\mathbb{R}^{r_{i-1}\times m\times r_i}:i=2,\cdots,5\}$.
		\Ensure Tensors $\tensor{\rho}^n:=\left({\bm{\rho}}^n_{\bm{k}}\right)$, $\tensor{U}^n:=\left(\tensor{U}^{1,n}_{\bm{k}},\tensor{U}^{2,n}_{\bm{k}},\tensor{U}^{3,n}_{\bm{k}}\right)^T$, and $\tensor{T}^n:=\left(\tensor{T}^n_{\bm{k}}\right)$.
		\State Compute a fourth-order tensor $\tensor{Y}^n:=\left(\tensor{Y}^n_{k_1,k_2,k_3,\alpha_3}\right)$ using \eqref{eq:temtensor}.
		\State Calculate the components of the full tensor $\bm{\rho}^{n}$ for the number density:
		\begin{equation}
			\label{eq:integral-rho-n1}
			\tensor{\rho}^n_{k_1,k_2,k_3} =h^3\sum\limits_{\alpha_3,\alpha_4,\alpha_5}\left(\tensor{Y}^n_{k_1,k_2,k_3,\alpha_3}\sum\limits_{l_1,l_2,l_3}\left(\tensor{G}^{4,n}_{\alpha_3,l_1,\alpha_4}\tensor{G}^{5,n}_{\alpha_4,l_2,\alpha_5}\tensor{G}^{6,n}_{\alpha_5,l_3}\right)\right).
		\end{equation}
		\State Obtain the components of the full tensor $\tensor{U}^{i,n}$ for the velocity:
		\begin{equation*}
			\begin{split}
				\tensor{U}_{k_1,k_2,k_3}^{i,n} = \frac{h^3}{\tensor{\rho}^n_{k_1,k_2,k_3}}\sum\limits_{\alpha_3,\alpha_4,\alpha_5}\left(\tensor{Y}^n_{k_1,k_2,k_3,\alpha_3}\sum\limits_{l_1,l_2,l_3}\left(v_i^{l_i}\tensor{G}^{4,n}_{\alpha_3,l_1,\alpha_4}\tensor{G}^{5,n}_{\alpha_4,l_2,\alpha_5}\tensor{G}^{6,n}_{\alpha_5,l_3}\right)\right).
			\end{split}
		\end{equation*}
		\State Compute  the components of the full tensor $\tensor{T}^n$ for the temperature:
		\begin{equation*}
			\begin{split}
				\tensor{T}_{k_1,k_2,k_3}^n &= \frac{\mathrm{Bo} h^3}{3\tensor{\rho}^n_{k_1,k_2,k_3}}\sum\limits_{\alpha_3,\alpha_4,\alpha_5}\Bigg(\tensor{Y}^n_{k_1,k_2,k_3,\alpha_3}\\
				&\qquad\sum\limits_{l_1,l_2,l_3}\left(\sum\limits_{i=1}^3(v_i^{l_i}-\tensor{U}_{k_1,k_2,k_3}^{i,n})^2\tensor{G}^{4,n}_{\alpha_3,l_1,\alpha_4}\tensor{G}^{5,n}_{\alpha_4,l_2,\alpha_5}\tensor{G}^{6,n}_{\alpha_5,l_3}\right)\Bigg).
			\end{split}
		\end{equation*}
	\end{algorithmic}
\end{algorithm}

\begin{algorithm}
	\normalsize
	\caption{Construction of the low-rank TT representation of $\tensor{F}_{\text{eq}}^{n}$.}
	\label{algo:rhs}
	\begin{algorithmic}[1]
		\Require TT format tensor $\tensor{F}^{n}$ and third-order full tensors $\tensor{\rho}^n$, $\tensor{U}^n$, and $\tensor{T}^n$.
		\Ensure TT cores of $\tensor{F}_{\text{eq}}^{n} $.
		\State Compute the fourth-order full tensor $\tensor{F}_{\text{eq}}^{i,n}:=([\tensor{F}_{\text{eq}}^{i,n}]_{k_1,k_2,k_3,l_i})$ with $ (i=1,2,3)$:
		\begin{equation*}
			\begin{split}
				[\tensor{F}_{\text{eq}}^{i,n}]_{k_1,k_2,k_3,l_i}&=\frac{(\tensor{\rho}^{n}_{k_1,k_2,k_3})^{1/3}}{{(2\pi \tensor{T}^{n}_{k_1,k_2,k_3}/\mathrm{Bo})^{1/2}}}\exp{\left(-\mathrm{Bo}\frac{({v}_{i}^{l_{i}}-\tensor{U}^{i,n}_{k_1,k_2,k_3})^{2}}{2\tensor{T}^{n}_{k_1,k_2,k_3}}\right)}.
			\end{split}
		\end{equation*}
		\State Using the classical TT-SVD algorithm to recompress the full tensor $\tensor{F}^{i,n}_{\text{eq}}\ (i=1,2,3)$ into a low-rank TT format tensor:
		\begin{equation*}
			\tensor{F}_{\text{eq}}^{i,n}\approx\mathbb{TT}(\tensor{F}_{\text{eq}}^{i,n}):=\mathbb{TT}(\tensor{G}^{1,i,n}_{\text{eq}},\tensor{G}^{2,i,n}_{\text{eq}},\tensor{G}^{3,i,n}_{\text{eq}},\tensor{G}^{4,i,n}_{\text{eq}}).
		\end{equation*}
		\State Convert $\mathbb{TT}(\tensor{F}_{\text{eq}}^{i,n})$ to the sixth-order tensor $\bm{\mathcal{E}}^{i, n}\ (i=1,2,3)$ as follows:
		\begin{enumerate}
			\item[(a)] Initialize $\tensor{E}^{i,n}:=\mathbb{TT}(\tensor{G}^{1,i,n}_{\text{eq}},\tensor{G}^{2,i,n}_{\text{eq}},\tensor{G}^{3,i,n}_{\text{eq}},\tensor{G}^{4,i,n}_{\text{eq}})$.
			\item[(b)] For $l = 1,2,3$, if $l \not= i$, then compute
			$$
			\tensor{E}^{i,n} = \texttt{expand}\left(\tensor{E}^{i,n}, l+3,m\right).
			$$
		\end{enumerate}
		\State Compute the TT format of $\tensor{F}^n_{\mathrm{eq}}$ by
		\begin{equation}\label{feq-e1-e2-e3}
			\tensor{F}^n_{\mathrm{eq}} = \tensor{E}^{1,n} \odot \tensor{E}^{2,n} \odot \tensor{E}^{3,n}.
		\end{equation}
		
	\end{algorithmic}
\end{algorithm}

\subsection{TT-based low-rank linear solver}\label{sec:tt-solver}

Based on the low-rank TT format of $\tensor{L}$ and the right-hand side $\tensor{R}:=(\tensor{I}+\Delta t\tensor{L})\tensor{F}^{n-1} + 2\Delta t\tensor{Q}^n$, the large-scale linear system \eqref{eq:linear1} can be rewritten into a TT-based low-rank linear system as following:
\begin{equation}\label{TTlowranksystem}
	(\tensor{I}-\Delta t\tensor{L})\tensor{F}^{n+1} = \tensor{R},
\end{equation}
which will be efficiently solved using TT-based low-rank linear solvers. TT-MALS, also known as DMRG \cite{cichocki2016tensor,holtz2012alternating}, is one of the most popular TT-based low-rank linear iterative methods, which updates the TT cores of $\tensor{F}^{n+1}$ in an alternating iterative manner.
The TT-MALS solver reformulates the TT-based low-rank linear system \eqref{TTlowranksystem} as a minimization problem, i.e. $\min\limits_{\tensor{F}^{n+1}}\|(\tensor{I}-\Delta t\tensor{L})\tensor{F}^{n+1} - \tensor{R}\|_F$, then solves the minimization problem by an alternating least squares approach until termination criterion is satisfied. In this paper, the termination criterion is set as
$$
\|(\tensor{I}-\Delta t\tensor{L})\tensor{F}_s^{n+1} - \tensor{R}\|_F\leq \varepsilon_d,
$$
where $\tensor{F}_s^{n+1}$ is the solution at the $s$-th iteration of the TT-MALS solver and $\varepsilon_d$ is a prescribed tolerance.
Since each iteration of the TT-MALS solver only involves the contraction operation on TT tensors and the calculation of small-scale matrix SVD, it has high computational efficiency \cite{holtz2012alternating,oseledets2012solution,rohrig2023performance}.

Since a good initial guess can significantly improve the convergence of the TT-MALS solver, we utilize the solution of an explicit leap-frog scheme as the initial guess for the TT-MALS solver. For linear system \eqref{eq:linear1}, the initial guess is defined as follows:
\begin{equation}\label{eq:initialization}
	\begin{split}
		\tensor{F}_0^{n+1} &=2\Delta t(\tensor{L}\tensor{F}^{n}+\tensor{Q}^n)+\tensor{F}^{n-1},
	\end{split}
\end{equation}
which is obtained with two additions and one matrix-vector product in TT format and still maintains the low TT-rank structure. 
As the CNLF scheme is a two-step scheme, the right-hand side $\tensor{R}$ in \eqref{TTlowranksystem} depends on $\tensor{F}^n$ and $\tensor{F}^{n-1}$. In the first time step, the CNLF scheme is not applicable because $\tensor{F}^{-1}$ is not available.
To ensure the second-order accuracy of the CNLF scheme, a second-order explicit total variation diminishing Runge-Kutta (TVD-RK) scheme is used at time $t_1$ and $\tensor{F}^1$ is calculated as follows
\begin{equation}
	\label{eq:F1-1}
	\begin{split}
		\hat{\tensor{F}}^{1} &= \tensor{F}^0 + \Delta t(\tensor{L}\tensor{F}^0 + \tensor{Q}^0),\\
		\tensor{F}^1 &= \frac{1}{2}\tensor{F}^0+\frac{1}{2}\hat{\tensor{F}}^{1}+\frac{1}{2}\Delta t(\tensor{L}\hat{\tensor{F}}^{1}+\hat{\tensor{Q}}^{1} ).
	\end{split}
\end{equation}
In summary, the overall computational procedure of the APTT solver is concluded in \Cref{algo:framework}.

\begin{algorithm}
	\normalsize
	\caption{Accuracy-Preserved Tensor-Train (APTT) method for solving the Boltzmann-BGK equation.}
	\label{algo:framework}
	\begin{algorithmic}[1]
		\State Initialization:
		Compute the low-rank TT format tensor $\tensor{F}^{0}$ using the initial condition $f(\bm{x},\bm{v},0)$.
		\State Calculate the low-rank TT format tensor ${\tensor{F}}^{1}$ by the TVD-RK scheme \eqref{eq:F1-1}.
		\State For $n=2:1:\frac{t^\star}{\Delta t}-1$ \\
		\begin{enumerate}
			\item[(a)] Compute the low-rank TT format tensor $(\tensor{I}+\Delta t\tensor{L})\tensor{F}^{n-1}$.
			\item[(b)] Construct the low-rank TT format tensor $\tensor{Q}^{n}$ by \Cref{algo:rho-u-t} and \ref{algo:rhs}.
			\item[(c)] Recompress the low-rank TT format $\tensor{R}$ using the TT-rounding algorithm.
			\item[(d)] Solve the linear system
			$
			(\tensor{I}-\Delta t\tensor{L})\tensor{F}^{n+1} = \tensor{R} $
			by the TT-MALS solver with the initial guess $\tensor{F}_0^{n+1}$ defined in \eqref{eq:initialization}.
		\end{enumerate}
	\end{algorithmic}
\end{algorithm}

\section{Complexity and convergence analysis of the APTT solver}\label{analysis}

We will analyze the newly proposed APTT method in terms of complexity and accuracy in this section.
The computational cost of the APTT solver is presented to provide a better understanding of its effectiveness.
The numerical error of the proposed APTT algorithm is established and is bounded by several prescribed accuracy tolerances. Through a careful selection of tolerances in the APTT algorithm, the low-rank solutions of the Boltzmann-BGK equations converge to those of the original discrete system with an explicit convergence rate.

\subsection{Complexity analysis}\label{complexity_analysis}
Let $\bm{r}^n=(r_1^n,r_2^n,r_3^n,r_4^n,r_5^n)^T$ be the TT-rank of the TT format tensor $\tensor{F}^{n}$ and $ r=\max(\|\bm{r}^{n-1}\|_\infty,\|\bm{r}^n\|_\infty)$. In step 3(a) of \Cref{algo:framework}, we compute $(\tensor{I}+\Delta t\tensor{L})\tensor{F}^{n-1}$ based on the matrix-vector product in TT format, whose computational complexity and memory cost are $\mathcal{O}(m^2r^4)$.
Step 3(b) of \Cref{algo:framework} is to construct the low-rank TT format tensor $\tensor{Q}^n$ using \Cref{algo:rho-u-t} and \ref{algo:rhs}. The computational complexities and memory costs for each step of \Cref{algo:rho-u-t} are listed as follows.
\begin{itemize}
	\item[1:] Computational complexity and memory cost for constructing $\tensor{Y}^n$ are $\mathcal{O}(m^3r^2)$ and $\mathcal{O}(m^3r)$, respectively.
	\item[2-4:] Computational complexity and memory cost for the calculation of full tensor $\tensor{\rho}^n, \,\tensor{U}^n, \,\tensor{T}^n$ are $\mathcal{O}(m^3r^2)$  and $\mathcal{O}(m^3)$, respectively.
\end{itemize}

The computational complexities and memory costs for step 1-4 of \Cref{algo:rhs} are summarized as follows.
\begin{itemize}
	\item[1:] Computational complexity and memory cost for calculation of $\tensor{F}_{\text{eq}}^{i,n}$ both are $\mathcal{O}(m^4)$.
	\item[2:] Using the classical TT-SVD algorithm to recompress $\tensor{F}_{\text{eq}}^{i,n}$ into low-rank TT format tensor $\mathbb{TT}(\tensor{F}_{\text{eq}}^{i,n})$, the computational complexity and memory cost are $\mathcal{O}(m^4r+m^3r^2)$  and $\mathcal{O}(mr^2)$, respectively.
	\item[3:] Computational complexity and memory cost of expand operator to get $\tensor{E}^{i,n}$ with $i=1,2,3$ both are $\mathcal{O}(m)$.
	\item[4:] Computational complexity and memory cost for Hadamard product to obtain $\tensor{F}^n_{\mathrm{eq}}$ are $\mathcal{O}(mr^6)$.
\end{itemize}
The computational complexity and memory cost for computing $\tensor{Q}^n$ from \eqref{eq:expand-1} are $\mathcal{O}(m^3r^2)$ and $\mathcal{O}(mr^2)$, respectively. The total computational complexity and memory cost for step 3(b) of \Cref{algo:framework} are $\mathcal{O}(m^4r+m^3r^2+mr^6)$ and $\mathcal{O}(m^4+m^3r+mr^6)$, respectively. In step 3(c) of \Cref{algo:framework}, the computational complexity and memory cost of the TT-rounding algorithm are ${\cal O}(mr^3)$ and  ${\cal O}(mr^2)$, respectively.
{In step 3(d) of \Cref{algo:framework}, we use the TT-MALS solver to solve the involved linear system, which only requires contractions in TT format and small-scale linear system and matrix SVD computations in each iteration.
	{The corresponding computational complexity and memory cost for each iteration are $\mathcal{O}(m^2r^4)$ and $\mathcal{O}(m^2r^2)$, respectively.}
	
	{In summary, at each time step, the overall computational complexity and memory cost of the APTT solver are $\mathcal{O}(m^4r+m^3r^2+mr^6+m^2r^4 \# \mathrm{iter})$ and $\mathcal{O}(m^4+m^3r+m^2r^2+mr^4)$, where $\# \mathrm{iter}$ is the number of iterations of the TT-MALS solver. If $r\leq \sqrt{m}$, it is easy to see that APTT reduces the computational complexity and memory cost by two orders of magnitude in terms of $m$, i.e., from $\mathcal{O}(m^6)$ to $\mathcal{O}(m^4r)$.}
	
	\subsection{Convergence analysis}\label{error_analysis}
	According to the complexity analysis, the efficiency of the APTT algorithm strongly depends on the rank of the TT format. To suppress the growth of the TT-rank caused by linear algebra operations, it is necessary to frequently recompress the TT-rank of tensors in the APTT algorithm. However, this recompression step introduces additional errors, which are on the order of $\varepsilon_b$. Furthermore, the TT-based low-rank linear system is inexactly solved by the TT-MALS solver, which stops when the residual is less than the prescribed tolerance $\varepsilon_d$. Quantifying the impact of these errors on the convergence of solutions obtained by the APTT algorithm is an interesting and challenging task, especially during the time evolution process.

	Let us assume that tensors with tilde, i.e., $\tilde{\tensor{F}}^n$, $\tilde{\tensor{F}}^n_{\mathrm{eq}}$, and $\tilde{\tensor{Q}}^n$, are the exact solutions of \eqref{eq:cnlf}. We then establish the convergence analysis of the APTT algorithm based on the following assumptions.
	\begin{itemize}
		\item[(A1)] $  0<\underline{\rho}\leq\tilde{\tensor{\rho}}^n_{\bm{k}}$ and $0<\underline{T}\leq\tilde{\tensor{T}}^n_{\bm{k}}$. These assumptions correspond to the absence of absolute vacuum and absolute zero temperature in the system.
		\item[(A2)] $\|\tilde{\tensor{F}}^n\|_\infty\leq C$, which is a regularity requirement for the solution of \eqref{eq:cnlf}. As reported in \cite{duan2017global, duan2019boltzmann}, this regularity requirement is held for the Boltzmann equation. To our knowledge, there are no published results supporting this hypothesis for the discrete system \eqref{eq:cnlf}. However, if this hypothesis is not true, the solution of \eqref{eq:cnlf} will blow up, and it is not necessary to perform an error analysis.
		\item[(A3)] The time step size $\Delta t$ is small enough such that $\|(\tensor{I}-\Delta t\tensor{L})^{-1}\|_F=\kappa_1\leq C$ and $\|(\tensor{I}+\Delta t\tensor{L})\|_F=\kappa_2\leq C$.
		\item[(A4)] The constant $C$ in (A2), (A3), and the following of this paper is independent of the time step size $\Delta t$ and tolerances $\varepsilon_b$, $\varepsilon_d$.
	\end{itemize}
	
	The following theorem gives the error bound of the low-rank TT format collision term $\tensor{Q}^n$.
	
	\begin{thm}\label{thm:error-01}
		Based on the assumptions (A1-A4), let $\varepsilon_{b}$ be the tolerance of the classical TT-SVD algorithm. For small enough $h$ satisfying $h^{3/2}\varepsilon_n\leq \underline{\rho}/2$, we have
		\begin{equation}
			\|\tilde{\tensor{Q}}^n-\tensor{Q}^n\|_F\leq C(\varepsilon_n+\varepsilon_{b}),
		\end{equation}
		where $\varepsilon_n=\|\tilde {\tensor{F}}^n-{\tensor{F}}^n\|_F$ and $C$ is a constant independent of $\varepsilon_n$ and $\varepsilon_{b}$.
	\end{thm}
	\begin{proof}
		It follows from \eqref{eq:integral-rho} and \eqref{eq:integral-rho-n1} that
		\begin{equation}
			\label{eq:thm-1}
			\tilde{\tensor{\rho}}^n=h^3\tilde{\tensor{F}}^n\times_{4}\bm{e}\times_5\bm{e}\times_6\bm{e},\quad {\tensor{\rho}}^n=h^3{\tensor{F}}^n\times_{4}\bm{e}\times_5\bm{e}\times_6\bm{e},
		\end{equation}
		where $\bm{e}=(1,1,\cdots,1)\in\mathbb{R}^{1\times m}$ and $\times_{i}$ denotes mode-$i$ product operation. According to \eqref{eq:thm-1}, we have
		\begin{equation}
			\begin{split}
				\|\tilde {\tensor{\rho}}^n-{\tensor{\rho}}^n\|_F^2&=h^6\|\tilde{\tensor{F}}^n\times_{4}\bm{e}\times_5\bm{e}\times_6\bm{e}-{\tensor{F}}^n\times_{4}\bm{e}\times_5\bm{e}\times_6\bm{e}\|_F^2\\&\leq
				h^5\|\tilde{\tensor{F}}^n\times_{4}\bm{e}\times_5\bm{e}-{\tensor{F}}^n\times_{4}\bm{e}\times_5\bm{e}\|_F^2\\
				&\leq \cdots\leq h^3 \|\tilde{\tensor{F}}^n-{\tensor{F}}^n\|_F^2,
			\end{split}
		\end{equation}
		which implies that $\|\tilde {\tensor{\rho}}^n-{\tensor{\rho}}^n\|_F\leq C\varepsilon_n $ and $\tensor{\rho}_{\bm{k}}^n\geq\tilde{\tensor{\rho}}_{\bm{k}}^n-h^{3/2}\varepsilon_n\geq \underline{\rho}-h^{3/2}\varepsilon_n>0$ for small enough $h$. By setting $\bm{\bar{v}}=(v_1^0,v_1^1,\cdots,v_1^{m-1})$,  we get
		\begin{equation}
			\label{eq:thm-2}
			\begin{split}
				\|\tilde {\tensor{U}}^n\odot\tilde{\tensor{\rho}}^n-{\tensor{U}}^n\odot{\tensor{\rho}}^n\|_F^2&=h^6\|\tilde{\tensor{F}}^n\times_{4}\bm{\bar{v}}\times_5\bm{\bar{v}}\times_6\bm{\bar{v}}-{\tensor{F}}^n\times_{4}\bm{\bar{v}}\times_5\bm{\bar{v}}\times_6\bm{\bar{v}}\|_F^2\\
				&\leq Ch^3\|\tilde{\tensor{F}}^n-{\tensor{F}}^n\|_F^2.
			\end{split}
		\end{equation}
		Due to $\tensor{\rho}_{\bm{k}}^n\geq \underline{\rho}-h^{3/2}\varepsilon_n$ and $\tilde{\tensor{\rho}}_{\bm{k}}^n\geq \underline{\rho}$,
		it follows from  \eqref{eq:thm-2} that
		\begin{equation}
			\begin{split}
				\|\tilde {\tensor{U}}^n-{\tensor{U}}^n\|_F^2&= \|\tilde {\tensor{U}}^n\odot\tilde{\tensor{\rho}}^n\odot\left(\tilde{\tensor{\rho}}^n\right)^{\odot^{-1}}-{\tensor{U}}^n\odot{\tensor{\rho}}^n\odot\left({\tensor{\rho}}^n\right)^{\odot^{-1}}\|_F^2\\
				&\leq\|\tilde {\tensor{U}}^n\odot\tilde{\tensor{\rho}}^n\odot\left(\tilde{\tensor{\rho}}^n\right)^{\odot^{-1}}-\tilde{\tensor{U}}^n\odot\tilde{\tensor{\rho}}^n\odot\left({\tensor{\rho}}^n\right)^{\odot^{-1}}\|_F^2\\
				&~~+\|\tilde{\tensor{U}}^n\odot\tilde{\tensor{\rho}}^n\odot\left({\tensor{\rho}}^n\right)^{\odot^{-1}}-{\tensor{U}}^n\odot{\tensor{\rho}}^n\odot\left({\tensor{\rho}}^n\right)^{\odot^{-1}}\|_F^2\\
				&\leq \frac{\|\tilde{\tensor{U}}^n\odot\tilde{\tensor{\rho}}^n\|^2_\infty }{\underline{\rho}^2\underline{\rho}_0^2}\|\tilde {\tensor{\rho}}^n-{\tensor{\rho}}^n\|_F^2+\frac{1}{\underline{\rho}_0^2}\|\tilde {\tensor{U}}^n\odot\tilde{\tensor{\rho}}^n-{\tensor{U}}^n\odot{\tensor{\rho}}^n\|_F^2\\
				&\leq Ch^3\|\tilde{\tensor{F}}^n-{\tensor{F}}^n\|_F^2,
			\end{split}
		\end{equation}
		where $\underline{\rho}_0=\underline{\rho}-h^{3/2}\varepsilon_n$.
		In a similar way, we can prove the following inequalities
		\begin{equation}
			\label{eq:thm1-ff}
			\|\tilde {\tensor{T}}^n-{\tensor{T}}^n\|_F\leq C\varepsilon_n ,\quad \tensor{T}_{\bm{k}}^n>0, \quad \|\tilde {\tensor{\nu}}^n-{\tensor{\nu}}^n\|_F\leq C\varepsilon_n.
		\end{equation}
		
		For the fourth-order full tensors $\tensor{F}^{i,n}_{\mathrm{eq}}$ with $i=1,2,3$ obtained in step 1 of \Cref{algo:rhs}, we get
		\begin{equation}
			\|\tilde {\tensor{F}}^{i,n}_{\mathrm{eq}}-{\tensor{F}}^{i,n}_{\mathrm{eq}}\|_F\leq C\left(\|\tilde {\tensor{\rho}}^n-{\tensor{\rho}}^n\|_F+\|{\tensor{U}}^n-{\tensor{U}}^n\|_F+\|\tilde{\tensor{T}}^n-{\tensor{T}}^n\|_F\right)\leq C\varepsilon_n.
		\end{equation}
		Then, the sixth-order tensor $\tensor{E}^{i,n}$ satisfies
		\begin{equation}
			\label{eq:thm1-ee}
			\|\tilde {\tensor{E}}^{i,n}-{\tensor{E}}^{i,n}\|_F\leq \|\tilde {\tensor{F}}^{i,n}_{\mathrm{eq}}-{\tensor{F}}^{i,n}_{\mathrm{eq}}\|_F+\|\mathbb{TT}(\tensor{F}_{\text{eq}}^{i,n})-{\tensor{F}}^{i,n}_{\mathrm{eq}}\|_F\leq C\varepsilon_n+\varepsilon_{b}.
		\end{equation}
		It follows form assumptions (A1) and (A2) that $\|\tilde {\tensor{E}}^{i,n}\|_\infty\leq C$. Combining this with \eqref{eq:thm1-ee} implies $\| {\tensor{E}}^{i,n}\|_\infty\leq C$.
		According to \eqref{feq-e1-e2-e3}, we have
		\begin{equation*}
			\begin{split}
				\|\tilde{\tensor{F}}_{\mathrm{eq}}^n-\tensor{F}_{\mathrm{eq}}^n\|_F\leq&\|\tilde{\tensor{E}}^{n,1}\odot\tilde{\tensor{E}}^{n,2}\odot(\tilde{\tensor{E}}^{n,3}-\tensor{E}^{n,3})\|_F+\|\tilde{\tensor{E}}^{n,1}\odot(\tilde{\tensor{E}}^{n,2}-\tensor{E}^{n,2})\odot\tensor{E}^{n,3}\|_F\\ &+\|(\tilde{\tensor{E}}^{n,1}-\tensor{E}^{n,1})\odot{\tensor{E}}^{n,2}\odot{\tensor{E}}^{n,3}\|_F\\
				\leq&\|\tilde{\tensor{E}}^{n,1}\odot\tilde{\tensor{E}}^{n,2}\|_\infty(C\varepsilon_n+\varepsilon_{b})+\|\tilde{\tensor{E}}^{n,1}\|_\infty\|{\tensor{E}}^{n,3}\|_\infty(C\varepsilon_n+\varepsilon_{b})\\
				&+\|{\tensor{E}}^{n,2}\|_\infty\|{\tensor{E}}^{n,3}\|_\infty(C\varepsilon_n+\varepsilon_{b})
				\leq C(\varepsilon_n+\varepsilon_{b}).
			\end{split}
		\end{equation*}
		Based on assumptions (A1) and (A2), we have $\|\tilde{\tensor{\nu}}^{n}\|_\infty\leq C$ and $ \|\tilde{\tensor{F}}^{n}_{\mathrm{eq}}\|_\infty\leq C$. Then, it follows from \eqref{eq:thm1-ff}, the inequality $\|{\tensor{\nu}}^{n}\|_\infty\leq C$ holds.
		Finally, due to \eqref{eq:expand-1}, we obtain
		\begin{equation*}
			\begin{split}
				\|\tilde{\tensor{Q}}^n&-\tensor{Q}^n\|_F\leq\|\frac{{\tensor{\nu}}^{n}}{\mathrm{Kn}}\odot(\tilde{\tensor{F}}^{n}_{\mathrm{eq}}-\tilde{\tensor{F}}^{n}-{\tensor{F}}^{n}_{\mathrm{eq}}+{\tensor{F}}^{n})\|_F+\|(\frac{\tilde{\tensor{\nu}}^{n}}{\mathrm{Kn}}-\frac{{\tensor{\nu}}^{n}}{\mathrm{Kn}})\odot(\tilde{\tensor{F}}^{n}_{\mathrm{eq}}-\tilde{\tensor{F}}^{n})\|_F\\
				\leq&\frac{1}{{\mathrm{Kn}}}\left(\|{{\tensor{\nu}}^{n}}\|_\infty(C\varepsilon_n+\varepsilon_{b}+\varepsilon_n)+\|\tilde{\tensor{F}}^{n}_{\mathrm{eq}}-\tilde{\tensor{F}}^{n}\|_\infty (C\varepsilon_n+\varepsilon_b)\leq C(\varepsilon_n+\varepsilon_{b})\right),
			\end{split}
		\end{equation*}
		which completes the proof of \Cref{thm:error-01}.
	\end{proof}

	The following theorem demonstrates that the APTT method maintains the same convergence rate as that of the discretization scheme by carefully setting tolerances.
	
	\begin{thm}\label{thm:error-02}
		Based on the assumptions (A1-A4), let $\varepsilon_{b}$ be the tolerance of the classical TT-SVD algorithm and $\varepsilon_{d}$ be the tolerance of the TT-MALS algorithm. For small enough $h$ satisfying $h^{3/2}\varepsilon_n\leq \underline{\rho}/2$, it holds that
		\begin{equation}
			\label{eq:thm2-01}
			\varepsilon_{n+1}\leq\varepsilon_{n-1}+C\Delta t (\varepsilon_{n-1}+\varepsilon_n+\varepsilon_b) + C(\varepsilon_b+\varepsilon_d),
		\end{equation}
		where $\varepsilon_n=\|\tilde {\tensor{F}}^n-{\tensor{F}}^n\|_F$ and $C$ is a constant independent of $\varepsilon_n$, $\varepsilon_{b}$, and $\varepsilon_{d}$. Furthermore, if we assume that $\varepsilon_b=\varepsilon_d=(\Delta t)^{1+\varpi}$ with $\varpi>0$,  we have
		\begin{equation}
			\label{eq:thm2-02}
			\varepsilon_{n+1}\leq C n(\Delta t)^{1+\varpi}\leq  C (\Delta t)^{\varpi}.
		\end{equation}
	\end{thm}
	
	\begin{proof}
		The right hand side $\tensor{R}$ of \eqref{TTlowranksystem} satisfies
		\begin{equation}
			\label{eq:assumptions-1}
			\begin{split}
				\tilde{\tensor{R}} - \tensor{R} & =\tilde{\tensor{R}}-(\tensor{I}+\Delta t\tensor{L})\tensor{F}^{n-1}-2\Delta t\tensor{Q}^n+{\tensor{R}}-(\tensor{I}+\Delta t\tensor{L})\tensor{F}^{n-1}-2\Delta t\tensor{Q}^n \\&=(\tensor{I}-\Delta t\tensor{L})(\tilde{\tensor{F}}^{n-1} -{\tensor{F}}^{n-1})+2\Delta t[\tensor{L}(\tilde{\tensor{F}}^{n-1} -{\tensor{F}}^{n-1})+\tilde{\tensor{Q}}^{n} -{\tensor{Q}}^{n}] \\
				&~~+{\tensor{R}}-(\tensor{I}+\Delta t\tensor{L})\tensor{F}^{n-1}-2\Delta t\tensor{Q}^n.
			\end{split}
		\end{equation}
		Let $\bar{\tensor{F}}^{n+1}$ be the exact solution of the following linear system
		\[
		(\tensor{I}-\Delta t\tensor{L}){\bar{\tensor{F}}^{n+1}}={\tensor{R}}.
		\]
		It follows from \eqref{eq:assumptions-1} and \Cref{thm:error-01} that
		\begin{equation}\label{eq:assumptions-2}
			\begin{split}
				\|\bar{\tensor{F}}^{n+1}&-\tilde{\tensor{F}}^{n+1}\|_F={\|(\tensor{I}-\Delta t\tensor{L})^{-1}(\tensor{R}-\tilde{\tensor{R}})\|_F}\\
				&\leq \|\tilde{\tensor{F}}^{n-1} -{\tensor{F}}^{n-1}\|_F+2\Delta t\|(\tensor{I}-\Delta t\tensor{L})^{-1}[\tensor{L}(\tilde{\tensor{F}}^{n-1} -{\tensor{F}}^{n-1})+\tilde{\tensor{Q}}^{n} -{\tensor{Q}}^{n}]\|_F\\
				&~~+\|(\tensor{I}-\Delta t\tensor{L})^{-1}[{\tensor{R}}-(\tensor{I}+\Delta t\tensor{L})\tensor{F}^{n-1}-2\Delta t\tensor{Q}^n]\|_F\\
				&\leq \varepsilon_{n-1}+2\Delta t\kappa_1[\kappa_2\varepsilon_{n-1}+C(\varepsilon_n+\varepsilon_b)]+\kappa_1\varepsilon_b\\
				&\leq \varepsilon_{n-1}+C\Delta t (\varepsilon_{n-1}+\varepsilon_n+\varepsilon_b) + C\varepsilon_b.
			\end{split}
		\end{equation}
		Due to the termination criterion of the  TT-MALS algorithm, we have
		\begin{equation}
			\label{eq:assumptions-3}
			\begin{split}
				\|\bar{\tensor{F}}^{n+1}-{\tensor{F}}^{n+1}\|_F &=\| (\tensor{I}-\Delta t \tensor{L})^{-1}\tensor{R}-{\tensor{F}}^{n+1}\|_F \\& \leq \|(\tensor{I}-\Delta t \tensor{L})^{-1}\|_F\| \tensor{R}-(\tensor{I}-\Delta t \tensor{L}){\tensor{F}}^{n+1}\|_F \leq
				\kappa_1 \varepsilon_d,
			\end{split}
		\end{equation}
		where $\varepsilon_d$ is the tolerance of the TT-MALS algorithm.
		Combining \eqref{eq:assumptions-1}, \eqref{eq:assumptions-2}, and \eqref{eq:assumptions-3},
		it implies
		\begin{equation}
			\label{eq:thm2-03}
			\begin{split}
				\varepsilon_{n+1}&=\|\tilde{\tensor{F}}^{n+1}-\tensor{F}^{n+1}\|_F \leq\|\bar{\tensor{F}}^{n+1}-{\tensor{F}}^{n+1}\|_F+\|\bar{\tensor{F}}^{n+1}-\tilde{\tensor{F}}^{n+1}\|_{F}\\
				&\leq \kappa_1\varepsilon_d +\|\bar{\tensor{F}}^{n+1}-\tilde{\tensor{F}}^{n+1}\|_{F}
				\leq\varepsilon_{n-1}+C\Delta t (\varepsilon_{n-1}+\varepsilon_n+\varepsilon_b) + C(\varepsilon_b+\varepsilon_d),
			\end{split}
		\end{equation}
		which completes the proof of \eqref{eq:thm2-01}. If  $\varepsilon_b=\varepsilon_d=(\Delta t)^{1+\varpi}$, it follows from \eqref{eq:thm2-03} that
		\begin{equation*}
			\begin{split}
				\varepsilon_{n+1}=&\leq\varepsilon_{n-1}+C\Delta t (\varepsilon_{n-1}+\varepsilon_n) + C(\Delta t)^{1+\varpi}\\
				&\leq \cdots \leq \exp\{Cn\Delta t  \}\varepsilon_0+Cn(\Delta t)^{1+\varpi}.
			\end{split}
		\end{equation*}
		Due to $n< \frac{t^\star}{\Delta t}$ and $\varepsilon_0=\varepsilon_b=(\Delta t)^{1+\varpi}$, we obtain that
		\begin{equation*}
			\begin{split}
				\varepsilon_{n+1}\leq Cn(\Delta t)^{1+\varpi}\leq C(\Delta t)^{\varpi},
			\end{split}
		\end{equation*}
		which completes the proof of \Cref{thm:error-02}.
	\end{proof}
	
	Based on \Cref{thm:error-02}, we have that the solution of the APTT solver satisfies the conservation laws of mass, momentum, and energy within the given tolerances.
	
	\begin{thm}[Conservation laws]\label{thm:error-04}
		The assumptions are the same as \Cref{thm:error-02}. The solution of the APTT algorithm satisfies the following conservation laws within the error bound $\varepsilon_n$.
		
		1) Conservation law of mass:
		\begin{equation}
			h^3|\left<\tensor{\rho}^n-\tilde{\tensor{\rho}}^0\right>|\leq C\varepsilon_n.
		\end{equation}
		2) Conservation law of momentum:
		\begin{equation}
			h^3|\left<\tensor{\rho}^n{\tensor{U}}^n-\tilde{\tensor{\rho}}^0\tilde{\tensor{U}}^0\right>|\leq C\varepsilon_n.
		\end{equation}
		3) Conservation law of energy:
		\begin{equation}
			h^3|\left<\tensor{W}^n-\tilde{\tensor{W}}^0\right>|\leq C\varepsilon_n,
		\end{equation}
		where energy density function ${\tensor{W}}^n_{\bm{k}}=\frac 1 2\sum\limits_{\bm{l}}{\tensor{F}}^n_{\bm{k},\bm{l}}\|\bm{v}^{\bm{l}}\|_2^2$ and $\left<\tensor{H}\right>=\sum\limits_{\bm{k}}\tensor{H}_{\bm{k}}$ for a tensor $\tensor{H}\in\mathbb{R}^{m\times m\times m}$.
	\end{thm}
	\begin{proof}
		It is easy to check that the discrete collision operator $\tilde{\tensor{Q}}^n$ satisfies conservation law \eqref{conservation-law}, i.e.
		\begin{equation}
			\sum\limits_{\bm{l}}\tilde{\tensor{Q}}^n_{\bm{k},\bm{l}}=0,\quad \sum\limits_{\bm{l}}\tilde{\tensor{Q}}^n_{\bm{k},\bm{l}}\bm{v}^{\bm{l}}=0,\quad \sum\limits_{\bm{l}}\tilde{\tensor{Q}}^n_{\bm{k},\bm{l}}\|\bm{v}^{\bm{l}}\|_2^2=0.
		\end{equation}
		Due to the periodic boundary conditions in $\bm{x}$ direction, we have
		\begin{equation}
			\sum\limits_{\bm{l}}[\bm{v}^{\bm{l}}\cdot\nabla_{\bm{x}}\tilde{\tensor{F}}^n_{\bm{k},\bm{l}}]_u=0,\quad
			\sum\limits_{\bm{l}}[\bm{v}^{\bm{l}}\cdot\nabla_{\bm{x}}\tilde{\tensor{F}}^n_{\bm{k},\bm{l}}]_u\bm{v}^{\bm{l}}=0,\quad
			\sum\limits_{\bm{l}}[\bm{v}^{\bm{l}}\cdot\nabla_{\bm{x}}\tilde{\tensor{F}}^n_{\bm{k},\bm{l}}]_u\|\bm{v}^{\bm{l}}\|^2_2=0.
		\end{equation}
		The exact solution of discrete system \eqref{eq:cnlf} satisfies the following conservation law of mass, momentum, and energy
		\begin{equation}
			\left<\tilde{\tensor{\rho}}^n-\tilde{\tensor{\rho}}^{n-2}\right>=2\Delta t \sum\limits_{\bm{k}}\sum\limits_{\bm{l}}\tilde{\tensor{R}}_{\bm{k},\bm{l}}=0,
		\end{equation}
		\begin{equation}
			\left<\tilde{\tensor{\rho}}^n\tilde{\tensor{U}}^n-\tilde{\tensor{\rho}}^{n-2}\tilde{\tensor{U}}^{n-2}\right>=\sum\limits_{\bm{k}}\sum\limits_{\bm{l}}\left(\tilde{\tensor{F}}^n_{\bm{k},\bm{l}}-\tilde{\tensor{F}}^{n-2}_{\bm{k},\bm{l}}\right)\bm{v}^{\bm{l}}=2\Delta t \sum\limits_{\bm{k}}\sum\limits_{\bm{l}}\tilde{\tensor{R}}_{\bm{k},\bm{l}}\bm{v}^{\bm{l}}=0,
		\end{equation}
		\begin{equation}
			\left<\tilde{\tensor{W}}^n-\tilde{\tensor{W}}^{n-2}\right>=\frac 1 2\sum\limits_{\bm{k}}\sum\limits_{\bm{l}}\left(\tilde{\tensor{F}}^n_{\bm{k},\bm{l}}-\tilde{\tensor{F}}^{n-2}_{\bm{k},\bm{l}}\right)\|\bm{v}^{\bm{l}}\|_2^2=\Delta t \sum\limits_{\bm{k}}\sum\limits_{\bm{l}}\tilde{\tensor{R}}_{\bm{k},\bm{l}}\|\bm{v}^{\bm{l}}\|_2^2=0,
		\end{equation}
		where $\tilde{\tensor{R}}:=- \frac{[\bm{v}^{\bm{l}}\cdot\nabla_{\bm{x}}\tilde{\tensor{F}}^{n}_{\bm{k},\bm{l}}]_u+[\bm{v}^{\bm{l}}\cdot\nabla_{\bm{x}}\tilde{\tensor{F}}^{n-2}_{\bm{k},\bm{l}}]_u}{2} + \tilde{\tensor{Q}}^{n-1}_{\bm{k},\bm{l}}$.
		According to $\|\tilde{\tensor{\rho}}^n-{\tensor{\rho}}^n\|_F\leq C\varepsilon_n$, we get the conservation law of mass for solution $\tensor{F}^n$
		\begin{equation}
			h^3\left|\left<{\tensor{\rho}}^n-\tilde{\tensor{\rho}}^{0}\right>\right|=h^3\left|\left<{\tensor{\rho}}^n-\tilde{\tensor{\rho}}^{n}\right>\right|\leq\|\tilde{\tensor{\rho}}^n-{\tensor{\rho}}^n\|_F\leq C\varepsilon_n.
		\end{equation}
		The conservation law of momentum and energy can be obtained similarly. Then, the proof of \Cref{thm:error-04} is completed.
	\end{proof}
	
	According to \Cref{thm:error-02}, by setting $\varepsilon_b=\varepsilon_d=(\Delta t)^{1+\varpi}$, the convergence rate of the APTT algorithm is $\varpi$. Since the CNLF scheme used in \eqref{eq:cnlf} is second-order, the newly proposed low-rank APTT solver is also second-order by setting $\varpi=3$. This will be verified by numerical simulations performed in \Cref{sec_res}. However, as discussed in subsection \ref{complexity_analysis}, it is crucial to derive bounds on the rank of $\tensor{F}^{n+1}$ for complexity analysis. Since we use the classical TT-SVD algorithm to find the best $\varepsilon_b$-approximation for recompression, the rank of $\tensor{F}^{n+1}$ is usually not the best. As reported in \cite{bachmayr2012adaptive, bachmayr2015adaptive, bachmayr2017kolmogorov}, linear iteration methods involving the best $\varepsilon_b$-approximation usually maintain quasi-optimal ranks, which remain of a similar size to these best approximation ranks. To study the rank of $\tensor{F}^{n+1}$, let us define the maximum TT-ranks of best approximations with error at most $\eta>0$ for tensor $\tensor{A}\in\mathbb{R}^{n_1\times n_2 \cdots\times n_d}$,
	\begin{equation}
		r_{\mathrm{best}}(\tensor{A},\eta)=\min\{r|\exists \tensor{B}\in\mathbb{R}^{n_1\times n_2 \cdots\times n_d}~\hbox{s.t.}~\|\mathrm{rank}(\tensor{B})\|_\infty\leq r ~\hbox{and } \|\tensor{A}-\tensor{B}\|_F\leq \eta\}.
	\end{equation}
	If the tensor $ \tensor{F}^{n+1}$ obtained by the APTT algorithm has relatively large TT-ranks, we then apply the classical TT-SVD algorithm to find the best $\varepsilon_b$-approximation and denote the solution as $\tensor{F}^{n+1}_{\varepsilon_b}$. The quasi-optimal ranks of $\tensor{F}^{n+1}_{\varepsilon_b}$ is guaranteed by the following theorem.

	\begin{thm}\label{thm:error-03}
		The assumptions are the same as \Cref{thm:error-02}. We have
		\begin{equation}
			\label{eq:thm:error-03-01}
			\|\tilde{\tensor{F}}^{n+1}-\tensor{F}^{n+1}_{\varepsilon_b}\|_F\leq \varepsilon_{n+1} + \varepsilon_b,\quad
			\|\mathrm{rank}(\tensor{F}^{n+1}_{\varepsilon_b})\|_\infty\leq r_{\mathrm{best}}(\tilde{\tensor{F}}^{n+1},\eta),
		\end{equation}
		where $\eta=\min(\varepsilon_{n+1},\frac{\varepsilon_b}{10})$.
	\end{thm}
	\begin{proof}
		The first inequality in \eqref{eq:thm:error-03-01} is straightforwardly obtained from \Cref{thm:error-02} and the best $\varepsilon_b$-approximation. The second inequality in \eqref{eq:thm:error-03-01} is a direct application of Lemma 5.4 in \cite{bachmayr2023low} (page 75), where we set $\kappa_{\mathbb{E}}=2D-1=5$, $\alpha=1$, $\mathbf{u}=\tilde{\tensor{F}}^{n+1}$, $\mathbf{v}={\tensor{F}}^{n+1}$, and $\eta=\min(\varepsilon_{n+1},\frac{\varepsilon_b}{10})$.
	\end{proof}
	
	\section{Numerical experiments}\label{sec_res}
	In this section, we perform several numerical experiments to validate the effectiveness and accuracy of the proposed APTT method in solving the Boltzmann-BGK equation. The tensor operations involved in the APTT method are implemented using the TT-Toolbox \cite{oseledets2011matlab}. Three test cases are studied, including a trigonometric initial value problem, a three-dimensional relaxation to statistical equilibrium problem, and a three-dimensional diffusion problem with discontinuous initial density. For comparison, we implement the GMRES algorithm to solve the matrix-vector linear system \eqref{eq:cnlf} and use its solution as the reference solution. All numerical experiments are conducted in MATLAB R2019b on a server with two NUMA nodes, each equipped with an 18-core Intel Xeon Gold CPU running at 2.60 GHz and 256GB DDR3 DRAM.
	
	\subsection{A trigonometric initial value problem}
	The initial PDF $f_0(\bm{x}, \bm{v})$ for this test case is taken as the Maxwellian distribution with macroscopic variables given by
	\begin{equation*}
		\rho_0(\bm{x}) = 1 + 0.5 \prod_{i=1}^D \sin(x_{i}), \quad \bm{U}_0(\bm{x}) = \bm{0}, \quad T_0(\bm{x}) = 1, \quad D=2\hbox{ or~} 3.
	\end{equation*}
	The tolerances in the APTT algorithm and the parameters of this test case are presented in \Cref{table:para-wave2d}.
	The relative error for the APTT algorithm is defined as
	$$\text{Relative error} = \frac{\norm{f_{\text{tt}} - f_{\text{ref}}}{2}}{\norm{f_{\text{ref}}}{2}},$$
	where $f_{\text{ref}}$ is the reference solution obtained by the traditional matrix-vector GMRES solver and $f_{\text{tt}}$ is the low-rank solution of the APTT algorithm.

	\begin{table}
		\centering
		\caption{Parameter setting of Boltzmann-BGK equation in the trigonometric initial value problem.}
		\begin{tabular}{ccc}
			\toprule
			Variable & Value & Description \\
			\midrule
			$\mathrm{K}$ & 1.0 & Collision frequency pre-factor \\
			$\mu$ & 0.5 & Collision frequency temperature exponent \\
			$\mathrm{Kn}$ & 1 & Knudsen number \\
			$\mathrm{Bo}$ & 3.65 & Boltzmann number \\
			$\varepsilon_{b}=\varepsilon_{d}$ & $10^{-6}$ & Tolerances in the APTT algorithm\\
			$\Delta t$ & 0.01 & Time step size\\
			\bottomrule
		\end{tabular}
		\label{table:para-wave2d}
	\end{table}

	\begin{figure}
		\centering
		\includegraphics[width=0.45\linewidth]{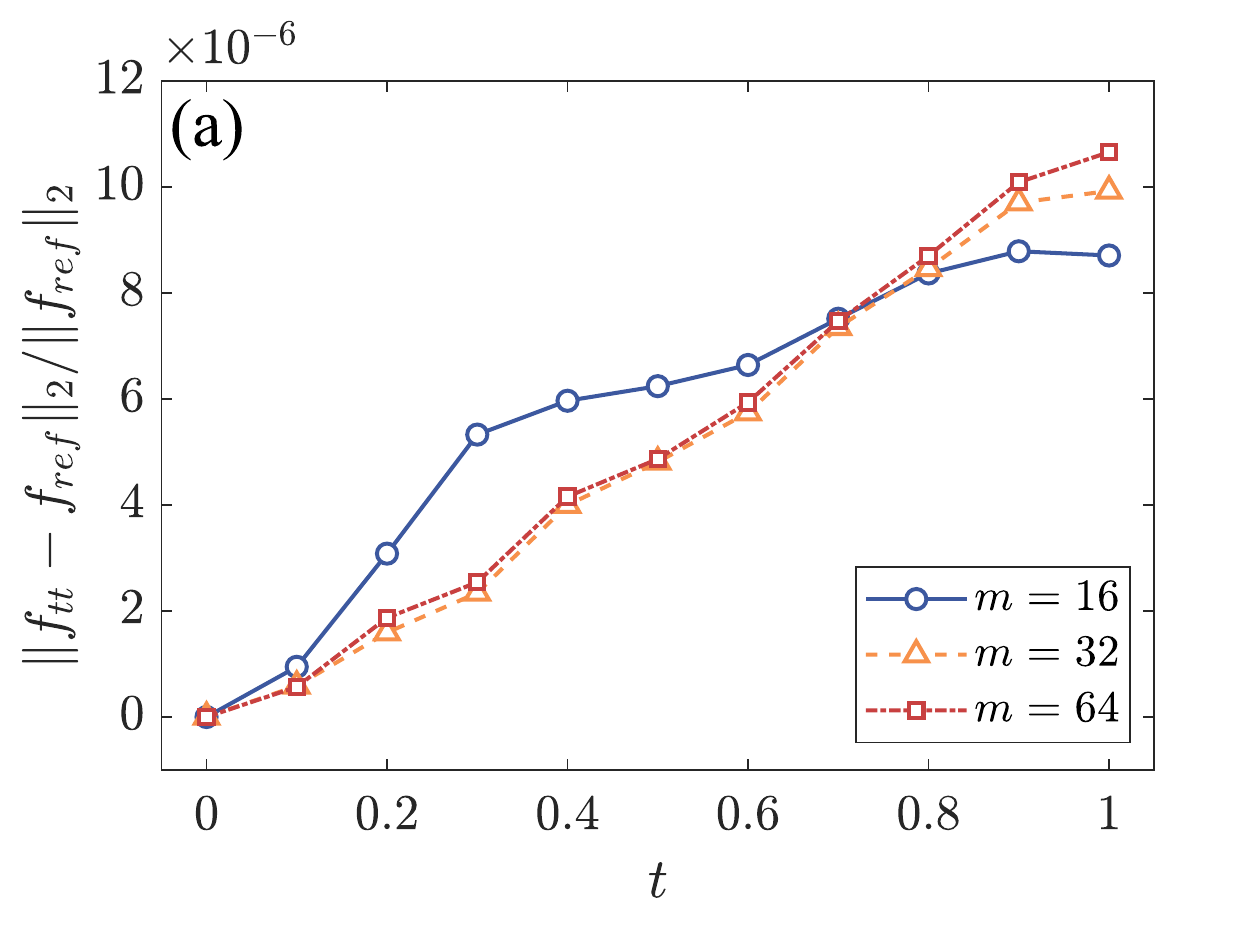}
		\includegraphics[width=0.45\linewidth]{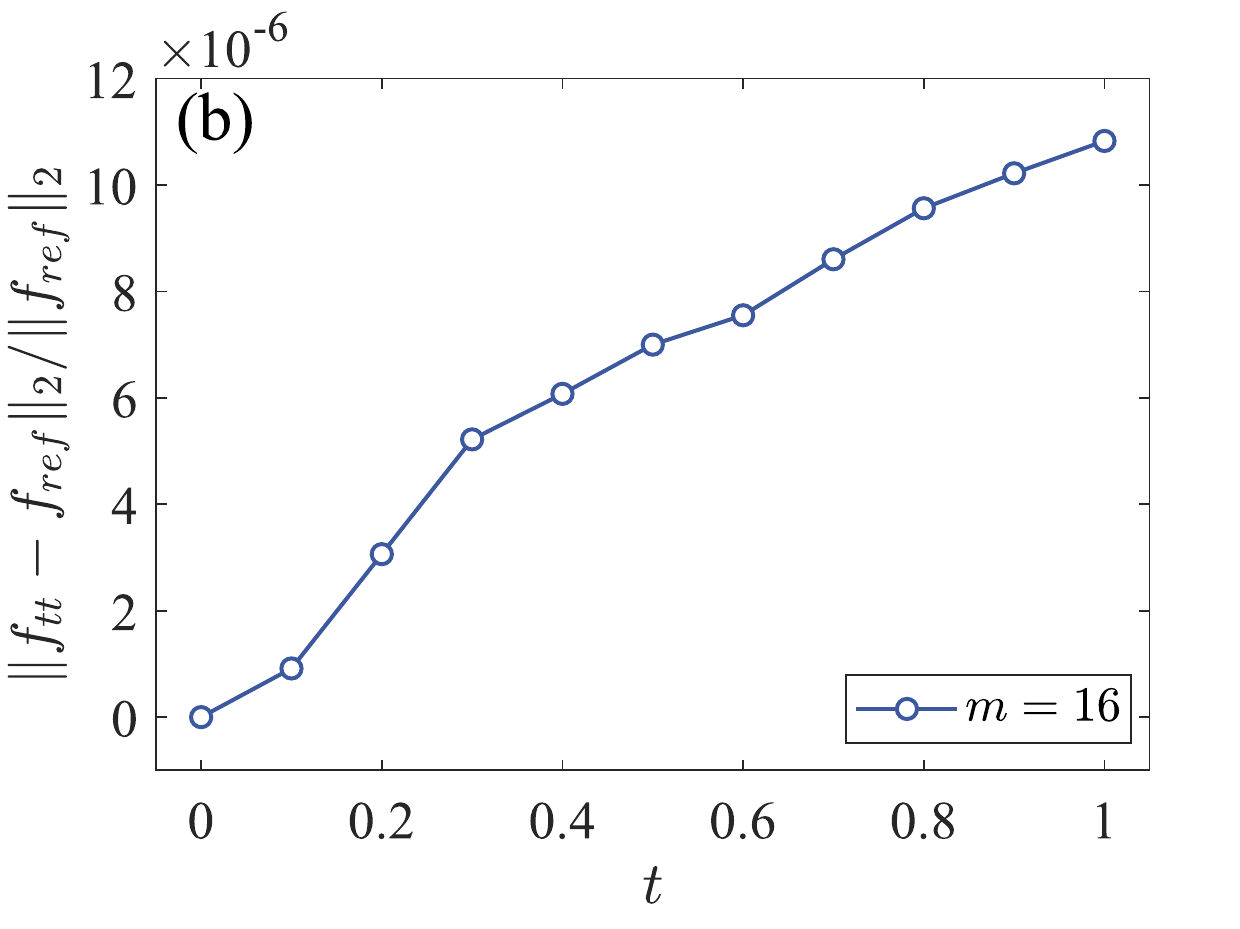}
		\caption{Relative errors for the trigonometric initial value problem. Left: two-dimensional case. Right: three-dimensional case.}
		\label{fig:accurate-compare}
	\end{figure}
	
	We begin by comparing the accuracy performance of the APTT algorithm and the traditional GMRES method. In the case of $D=2$, we conduct simulations for $m=16$, 32, and 64. For $D=3$, we present a simulation with $m=16$. Due to memory limitations of the traditional GMRES method, a comparison with a larger $m$ for $D=3$ is not provided.
	The relative error between the solutions obtained by the two methods with respect to time $t$ is displayed in \Cref{fig:accurate-compare}.  From this figure, we observe that the relative errors for all simulations are in the same order of tolerances $\varepsilon_b$ in the APTT algorithm.
	The macroscopic density, velocity, and temperature of the low-rank solution and reference solution for $D=2$ on a uniform mesh grid with $m=64$ are presented in \Cref{fig:wave-contour}. This indicates that the low-rank solution and reference solution are visually indistinguishable.
	As discussed in \Cref{error_analysis}, the error of the APTT algorithm comes from two main aspects: the approximation error of the low-rank TT format and the discretization error of the finite difference scheme. Since the two methods use the same finite difference scheme, the results in \Cref{fig:accurate-compare} indicate that the approximate error introduced by the low-rank TT format is well controlled by the predefined tolerance in \Cref{table:para-wave2d}. This observation verifies the analysis result \eqref{eq:thm2-01} in \Cref{thm:error-02} and is consistent with motivation for the best $\varepsilon_b$-approximation for the low-rank TT format.

	\begin{figure}
		\centering
		\includegraphics[width=1.00\linewidth]{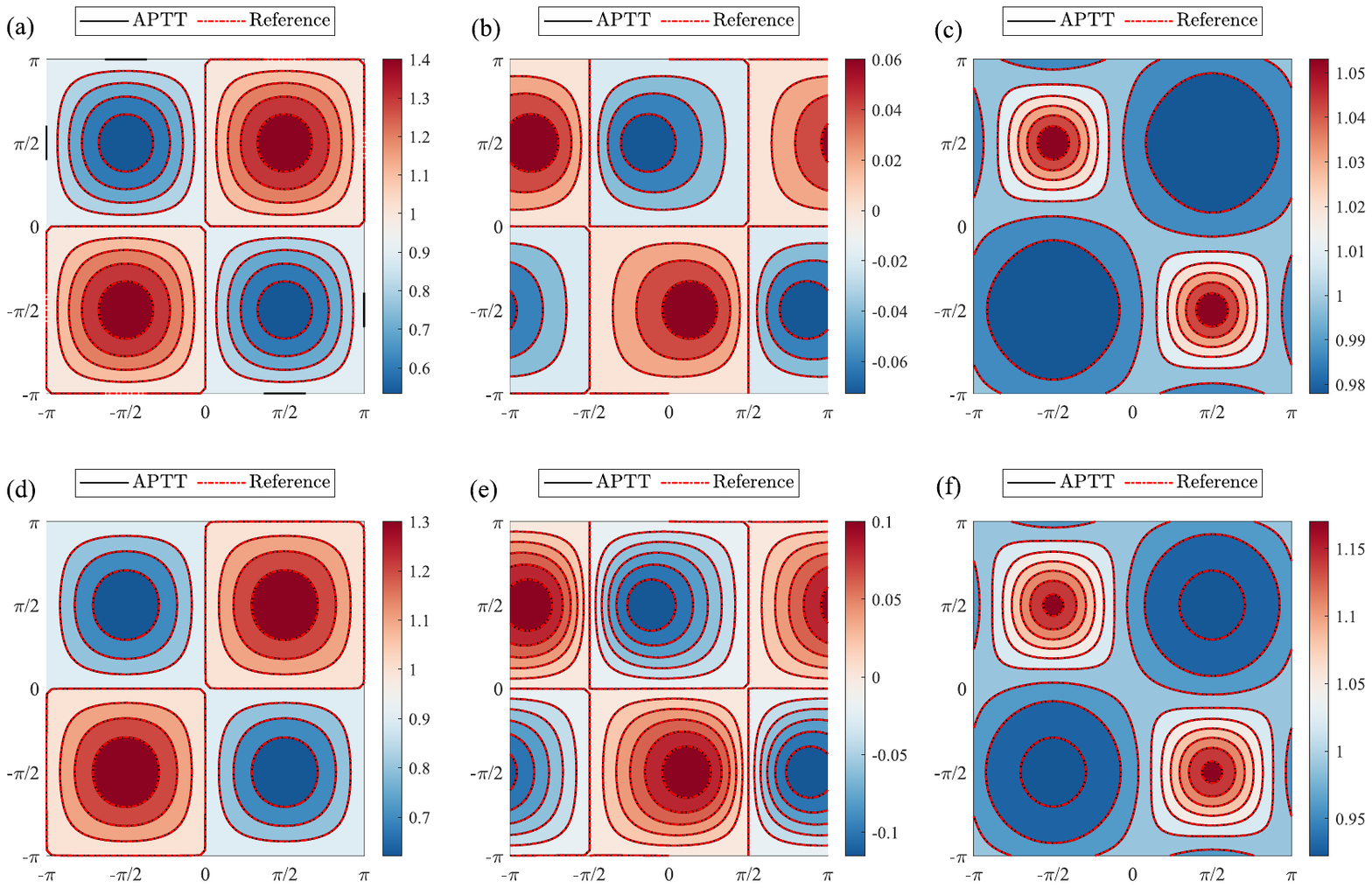}
		\caption{
			The contour plots of the macroscopic variables for the two-dimensional trigonometric initial value problem. The top and bottom rows represent the solutions at $t=0.5$ and $t=1.0$, respectively. The columns, from left to right, correspond to the macroscopic density, velocity on the first axis, and temperature.}
		\label{fig:wave-contour}
	\end{figure}
	
	Next, we investigate the convergence rate of the APTT method to validate the error analysis result \eqref{eq:thm2-02} presented in \Cref{thm:error-02}. Since the CNLF and upwinding schemes both are second-order, we run two-dimensional simulations with time step size $\Delta t= 1/(4m)$ and uniform mesh grid of mesh size $2\pi/m$, where $m$ takes values of 8, 16, 32, 64, and 128, respectively. The tolerances in the APTT algorithm are set to $\varepsilon_b = \varepsilon_d = 10^{-7}$ to ensure $\varepsilon_b = \varepsilon_d = {\cal O}(\Delta t)^3$. The solution obtained by the APTT method on a refined uniform mesh grid with $m=256$ is taken as the reference solution.
	For $t=1$, the $L^2$-norm errors of the solutions obtained by the APTT method with respect to $m$ are illustrated in \Cref{table:order}. It can be observed that the convergence rate of the APTT method is 2, aligning with the error analysis result \eqref{eq:thm2-02}. Thus, by carefully selecting tolerances, the APTT method can preserve the convergence rate of the given discrete scheme.

	\begin{table}[H]
		\centering
		\renewcommand\arraystretch{0.9}
		\addtolength{\tabcolsep}{5.0pt}
		\caption{$L^2$-norm errors with respect to $m$ for trigonometric initial value problem at $t = 1$.}
		\setlength{\tabcolsep}{0.5mm}{\begin{tabular}{c|ccccc}
				\toprule
				$m$ & 8 & 16 & 32 & 64 & 128\\
				\midrule
				$L^2$-norm error & $2.06 \times 10^{-3}$ & $4.78 \times 10^{-4}$ & $1.17 \times 10^{-4}$ & $2.76 \times 10^{-5}$ & $5.57 \times 10^{-6}$ \\
				Order & -- & 2.11 & 2.04 & 2.08 & 2.31\\
				\bottomrule
		\end{tabular}}
		\label{table:order}
	\end{table}

	The comparison of computing times and storage costs between the traditional GMRES solver and the APTT algorithm is presented in \Cref{table:wave2dtime}. To ensure fairness, both methods are set with tolerances of $10^{-6}$. As reported in \Cref{table:wave2dtime}, the computing time and memory cost of the GMRES solver increase with the fourth power of $m$ for $D=2$ and with the sixth power of $m$ for $D=3$, respectively. Due to the rapid increase in computational and memory costs, three-dimensional simulations with $m \geq 32$ using the traditional GMRES solver could not be performed on the given computer.
	The average rank of the low-rank TT format $\tensor{F}^n$ is also reported in \Cref{table:wave2dtime}, which is insensitive to $m$. Since the average rank is comparable to $m$ for $m \leq 64$, according to the computational complexity analysis in \Cref{thm:error-02} for $D=3$, the term $m^2r^4\#\text{iter}$ dominates in the computational cost. This explains why the computing time of the APTT algorithm increases with the square of $m$ for $m \leq 64$. Similar conclusions can be drawn for memory costs. In summary, the APTT method significantly enhances computational efficiency and reduces memory costs compared to the traditional GMRES solver, particularly for fine meshes.

	\begin{table}[H]
		\centering
		\caption{ The trigonometric initial value problem: computing times and memory costs for the traditional GMRES solver (denoted as "Reference") and the APTT algorithm.}
		\begin{tabular}{cccccc}
			\toprule
			&Methods &  & $m=16$ & $m=32$ & $m=64$ \\
			\midrule
			
			\multirow{5}{*}{$D=2$}&\multirow{2}{*}{Reference} & Computing time (s) & 0.70 & 8.82 & 279.12 \\
			&                  & Memory usage (bit) & 15.89M & 259.01M & 4.09G\\
			\cline{2-6}
			&                     \multirow{3}{*}{APTT} & Computing time (s) & 18.24 & 33.94 & 128.22 \\
			&                 & Memory usage (bit) & 665.70K & 2.04M & 8.78M
			\\
			&          & Average rank & 13.67 & 12.33 & 12.33 \\
			\midrule
			\multirow{5}{*}{$D=3$} & \multirow{2}{*}{Reference} & Computing time (s) & 206.71 & -- & -- \\
			&       & Memory usage (bit) & 4.91G & -- & --\\
			\cline{2-6}
			&\multirow{3}{*}{APTT} & Computing time (s) & 145.15 & 497.99 & 2058.64 \\
			&   & Memory usage (bit)  & 5.38M & 32.57M & 406.45M \\
			&    & Average rank & 30.60 & 28.80 & 28.60 \\
			\bottomrule
		\end{tabular}
		\label{table:wave2dtime}
	\end{table}

	\subsection{Relaxation to statistical equilibrium}
	The relaxation to statistical equilibrium problem with $D=3$ is a benchmark for the Boltzmann-BGK equation also studied in  \cite{boelens2020tensor}.  Initially, the PDF of particles is set as the following non-Maxwellian distribution
	\begin{equation}
		f(\bm{x}, \bm{v}, 0) = f_1(\bm{x}, v_1)f_2(\bm{x}, v_2)f_3(\bm{x},v_3),
	\end{equation}
	where
	\begin{equation*}
		f_i(\bm{x}, v_i) = \frac{\sqrt[3]{\rho_0(\bm{x})}}{\sqrt{2\pi T_0(\bm{x})/\mathrm{Bo}}}\exp\left(-\frac{\mathrm{Bo}}{2T_0(\bm{x})}(U_{i,0}(\bm{x})-v_i)^4\right), \hbox{with }i=1,2,3.
	\end{equation*}
	The macroscopic variables are respectively defined as
	\begin{equation*}
		\begin{gathered}
			\rho_0(\bm{x}) = \prod_{i=1}^3 (1 + 0.5\cos(x_i)),\quad  T_0(\bm{x}) = 1 + 0.0025\cos(x_1), \\
			U_{1,0}(\bm{x}) = 1 + 0.025\sin(x_2 - 1),\quad  U_{2,0}(\bm{x}) = 0,\quad U_{3,0}(\bm{x}) = 0.025\sin(x_1-2).
		\end{gathered}
	\end{equation*}
	The Knudsen number is set to $\mathrm{Kn} = 10$, and the time step is set to $0.005$. The other parameters are the same as shown in \Cref{table:para-wave2d}. The computational domain is covered by a uniform grid with $m=64$.
	
	The evelution of the PDF $f(\bm{x}, \bm{v}, t)$ at $\bm{x} = \bm{0}$ is depicted in \Cref{fig:relaxation3d6}. We can observe from \Cref{fig:relaxation3d6} that the PDF at $\bm{x} = \bm{0}$ evolves from a non-Maxwellian distribution to a Maxwellian distribution, which suggests that the PDF converges to the statistical equilibrium.
	\Cref{fig:relaxation3d_nU} and \Cref{fig:relaxation3d_T} presents the density $\rho$, velocity $\bm{U}$, temperature $T$ at different $t$. It is evident that these macroscopic variables deviate from the initial distribution as time progresses.
	
	As stated in \Cref{thm:error-04}, the solution of the APTT algorithm satisfies the conservation laws of mass, momentum, and energy within the specified tolerances.
	To verify this, we plot the relative errors corresponding to these conservation laws in \Cref{fig:relaxation3d1}, from which we observe that the relative errors are all in the same order as  $\varepsilon_n$. Since the error $\varepsilon_n$ introduced by the low-rank APTT algorithm can be bounded by the given tolerances $\varepsilon_b$ and $\varepsilon_d$, reducing the tolerances can further improve the relative errors corresponding to these conservation laws. This should be an advantage compared to the low-rank tensor method proposed in \cite{boelens2020tensor}, which suffers from an amount of mass loss of about $2\%$  per unit time.
	
	\begin{figure}
		\centering
		\includegraphics[width=0.9\linewidth]{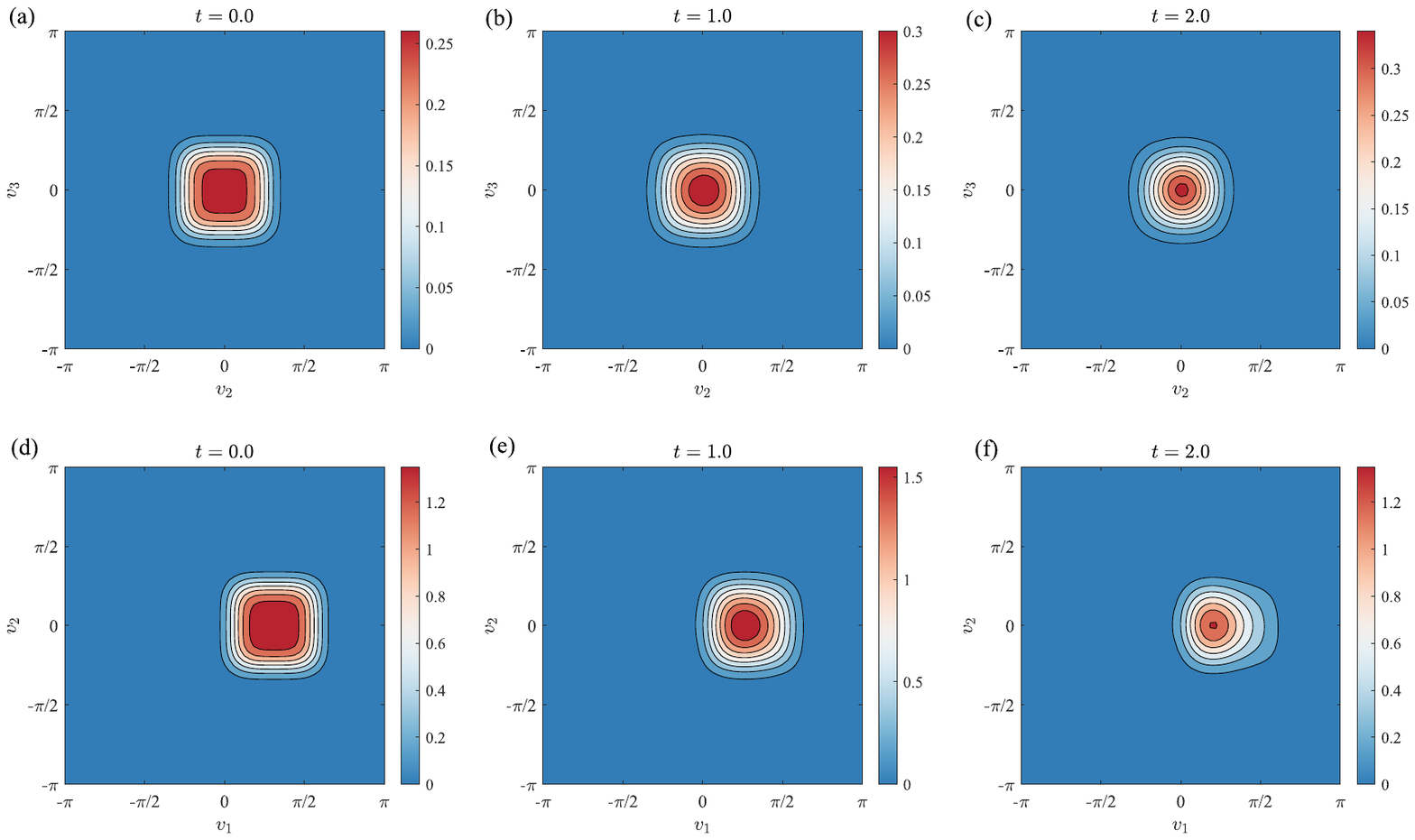}
		\caption{The relaxation to statistical equilibrium problem. Temporal evolution of $f(\bm{x},\bm{v},t)$ in hyper-plane $(\bm{x}, \bm{v}) = (\bm{0}, 0, v_2, v_3)$ and $(\bm{x}, \bm{v}) = (\bm{0}, v_1, v_2, 0)$.}
		\label{fig:relaxation3d6}
	\end{figure}
	
	\begin{figure}
		\centering
		\includegraphics[width=0.9\linewidth]{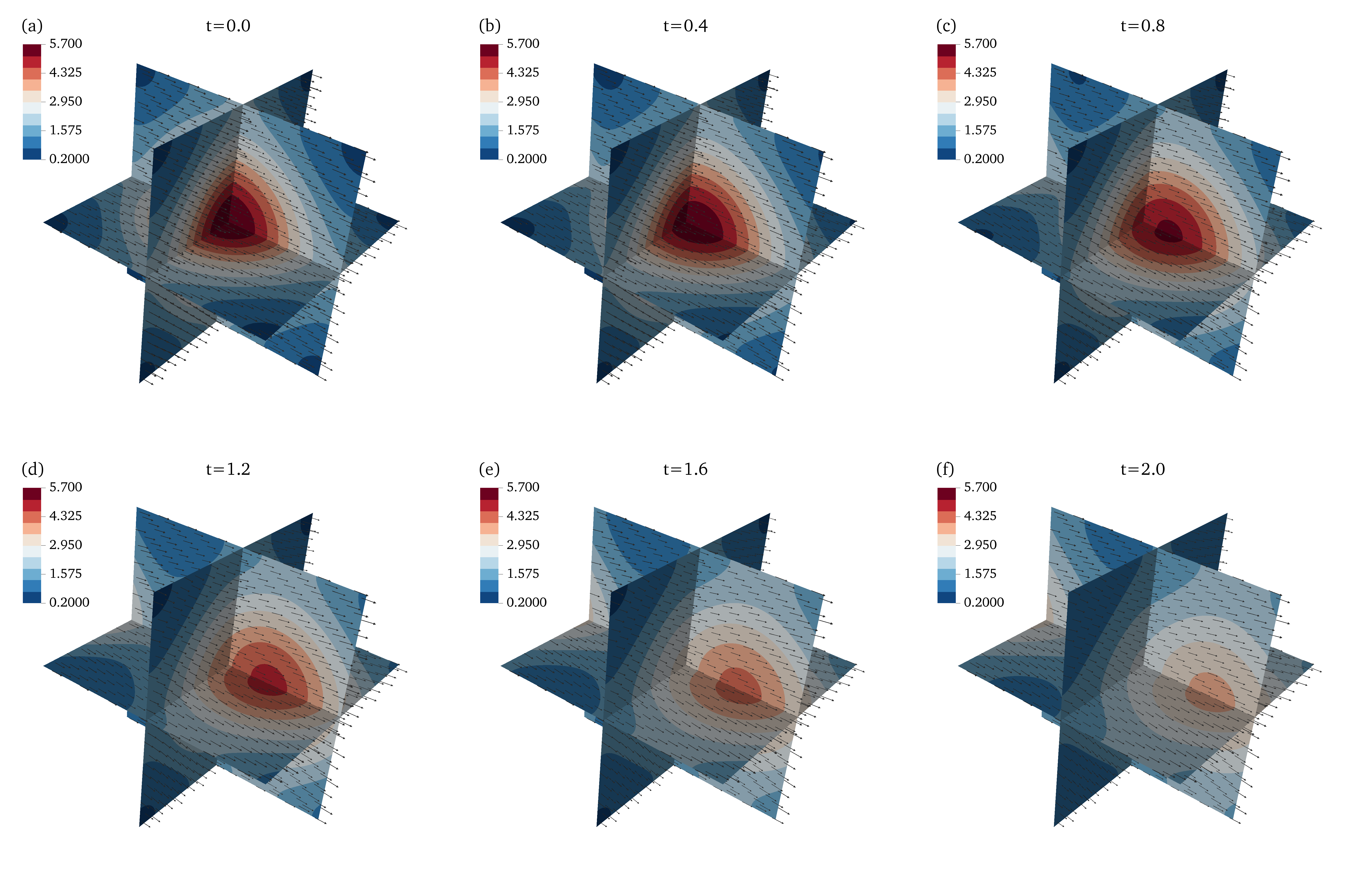}
		\caption{The relaxation to statistical equilibrium problem. Temporal evolution of density $\rho(\bm{x}, t)$ and velocity $\bm{U}(\bm{x}, t)$.}
		\label{fig:relaxation3d_nU}
	\end{figure}
	
	\begin{figure}
		\centering
		\includegraphics[width=0.9\linewidth]{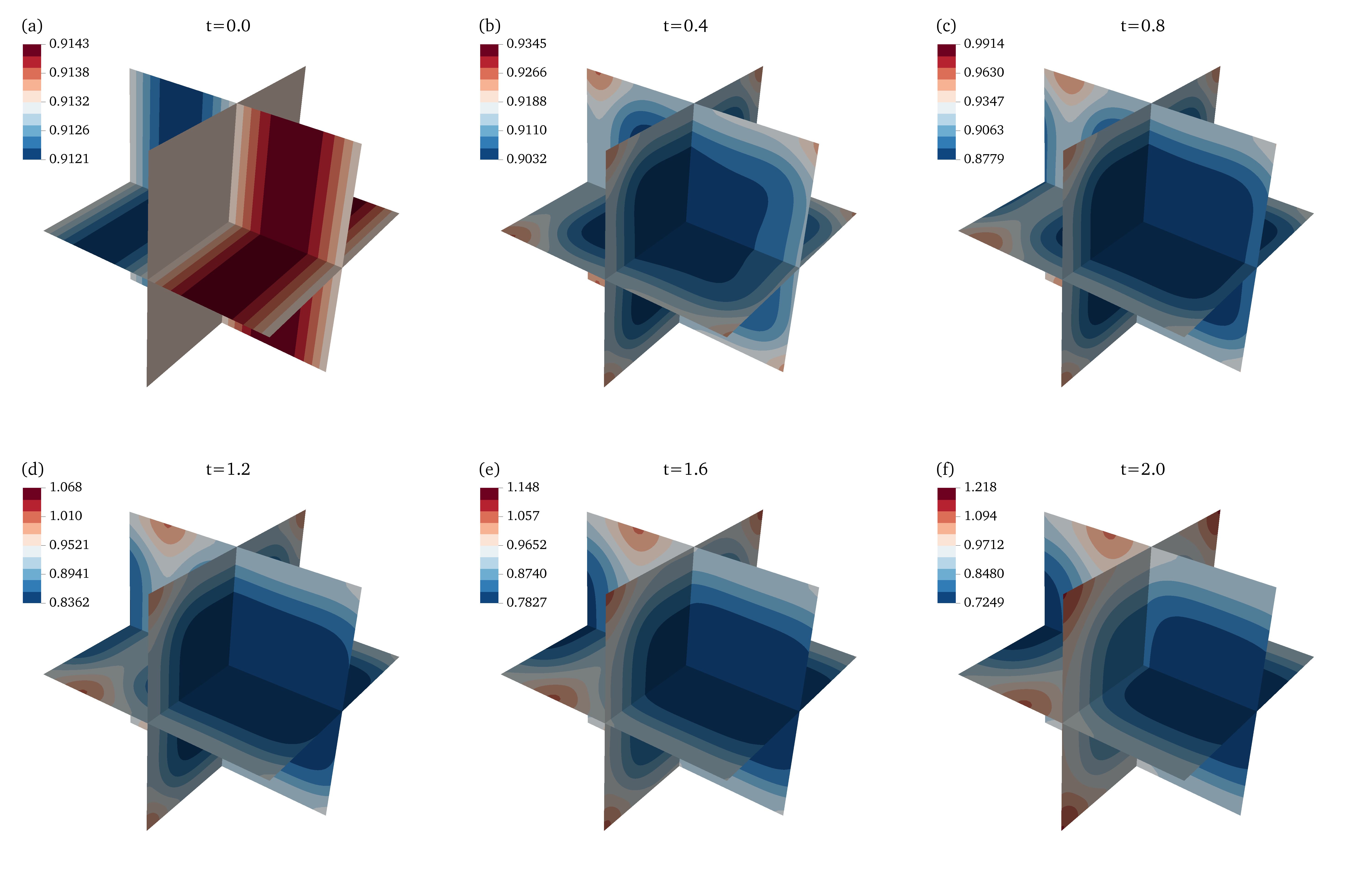}
		\caption{The relaxation to statistical equilibrium problem. Temporal evolution of temperature $T(\bm{x}, t)$.}
		\label{fig:relaxation3d_T}
	\end{figure}
	
	\begin{figure}
		\centering
		\includegraphics[width=1.00\linewidth]{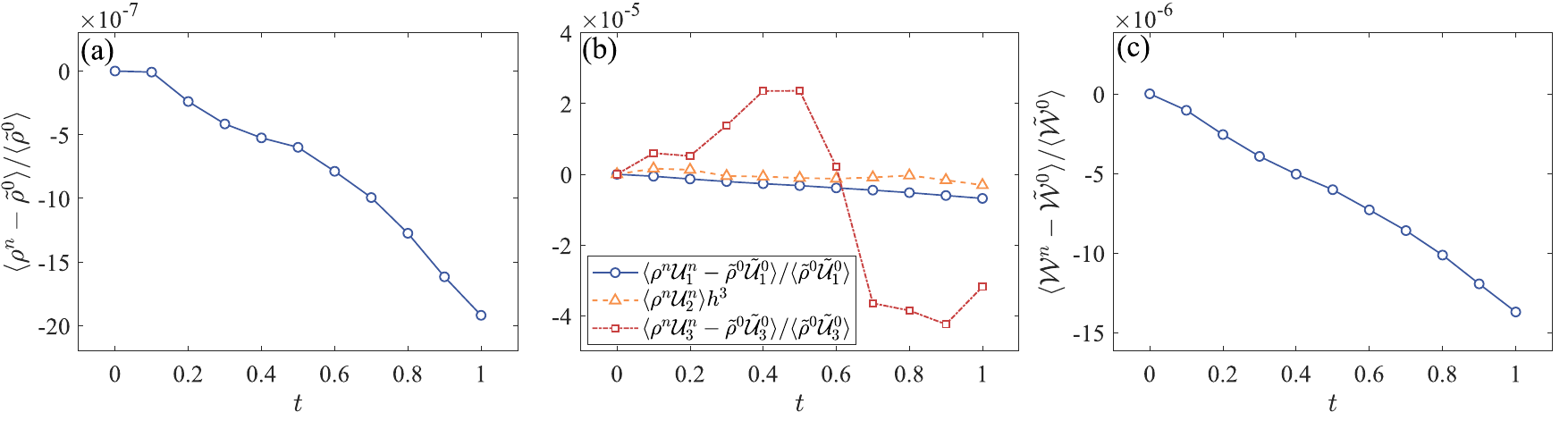}
		\caption{Verification of conservation laws for the relaxation to statistical equilibrium problem. From left to right: conservation law of mass, momentum, and energy. As $\langle\tilde{\tensor{\rho}}^0\tilde{\tensor{U}}^0_2\rangle=0$, the absolute error $\langle\tensor{\rho}^n\tensor{U}_2^n\rangle h^3$ in figure (b).}
		\label{fig:relaxation3d1}
	\end{figure}
	
	\subsection{A diffusion problem with discontinuous initial density}\label{sec_discontinuous_init}
	In the end, we consider a challenging problem with a discontinuous initial density. Specifically, the macroscopic variables at $t=0$ are given by
	\begin{equation*}
		\rho_0(\bm{x}) =
		\left\{
		\begin{array}{ll}
			10, & \max\limits_{i=1,2,3} |x_i| \leq \frac{\pi}{8}, \\
			1,  &\text{otherwise},
		\end{array}
		\right. \quad \bm{U}_0(\bm{x}) = \bm{0}, \quad T_0(\bm{x}) = 1.
	\end{equation*}
	The initial PDF $f(\bm x,\bm v,0)$ is set to the Maxwellian distribution with these macroscopic variables.
	In this test case, the Knudsen number, the tolerances in the APTT algorithm, and the time step size are set as $\mathrm{Kn}=10$, $\varepsilon_b=\varepsilon_d=10^{-5}$, and 0.002, respectively. The other parameters are the same as given in \Cref{table:para-wave2d}.
	To suppress numerical oscillations, we use the discrete system  \eqref{eq:modified_upwind} with artificial dissipation to solve this problem, where $\epsilon$ is equal to $0.1$. A uniform mesh grid with $m=64$ is employed.
	The relative errors corresponding to the conservation laws of mass, momentum, and energy are reported in \Cref{fig:diffusion3d1}, and they are in the same order as $\varepsilon_n$, verifying \Cref{thm:error-04}.
	The evolution of macroscopic variables, including density $\rho(\bm{x}, t)$, velocity $\bm{U}(\bm{x}, t)$ and temperature $T(\bm{x}, t)$ is exhibited in \Cref{fig:diffusion3d_nU} and \Cref{fig:diffusion3d_T}.

	The numerical results demonstrate that the low-rank APTT algorithm produces stable and accurate solutions for the diffusion problem, even in the presence of a discontinuous initial density. Although an artificial dissipation term is introduced in the upwinding scheme to mitigate numerical oscillations, some weak oscillations persist in the numerical solutions shown in \Cref{fig:diffusion3d_nU} and \Cref{fig:diffusion3d_T}. To further suppress these weak numerical oscillations, the APTT algorithm should be combined with a more stable scheme such as the essentially nonoscillatory (ENO) or weighted ENO (WENO) scheme, which is listed as future work.
	
	\begin{figure}
		\centering
		\includegraphics[width=1.00\linewidth]{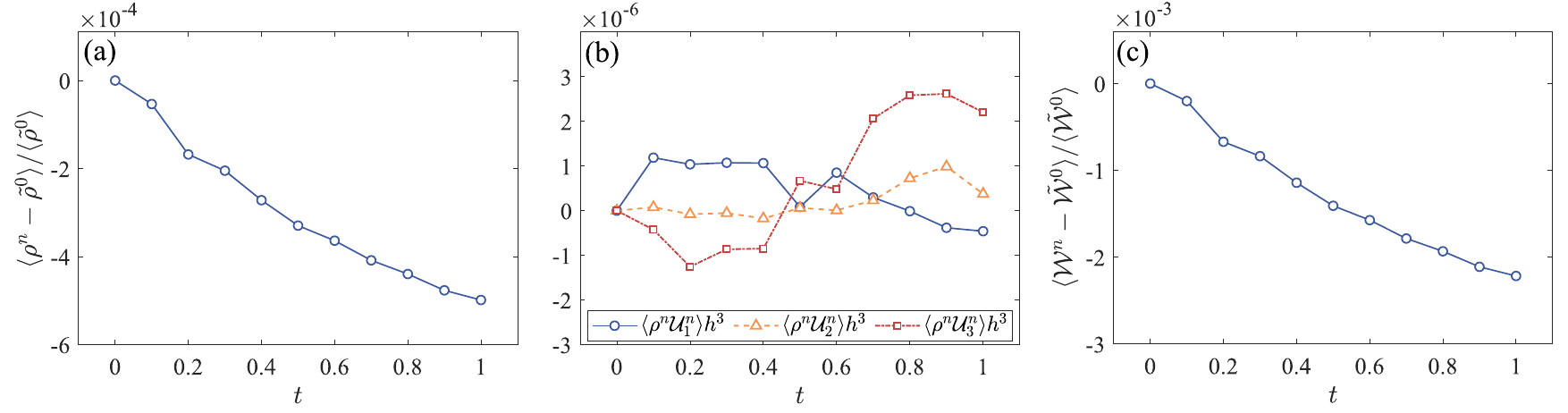}
		\caption{Verification of conservation laws for the diffusion problem with discontinuous initial density. From left to right: conservation law of mass, momentum, and energy. As $\langle\tilde{\tensor{\rho}}^0\tilde{\tensor{U}}^0_i\rangle=0$, the absolute error $\langle\tensor{\rho}^n\tensor{U}_i^n\rangle h^3$ in figure (b).}
		\label{fig:diffusion3d1}
	\end{figure}

	\begin{figure}
		\centering
		\includegraphics[width=0.9\linewidth]{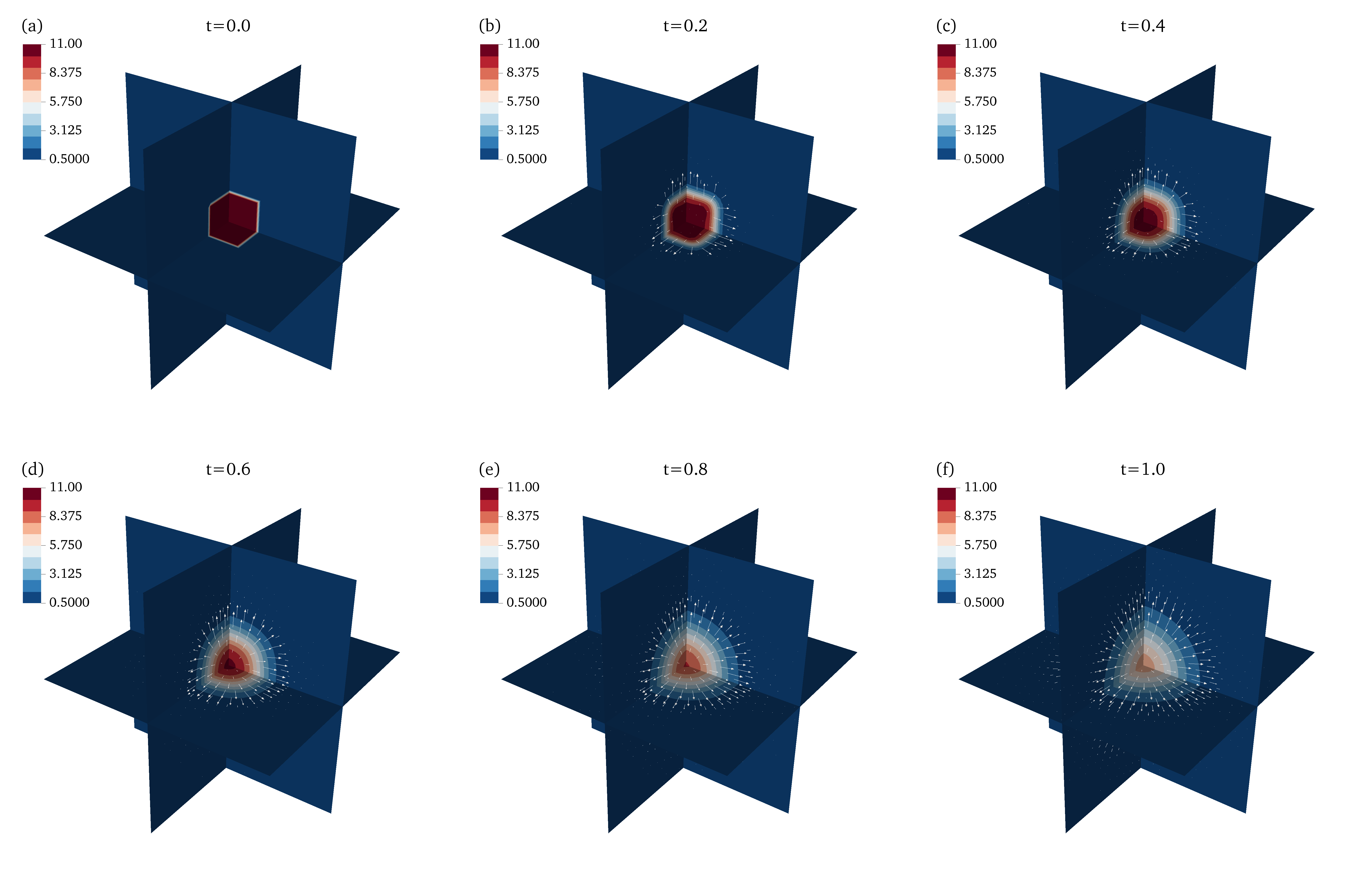}
		\caption{The diffusion problem with discontinuous initial density. Temporal evolution of density $\rho(\bm{x}, t)$ and velocity $\bm{U}(\bm{x}, t)$.}
		\label{fig:diffusion3d_nU}
	\end{figure}
	
	\begin{figure}
		\centering
		\includegraphics[width=0.9\linewidth]{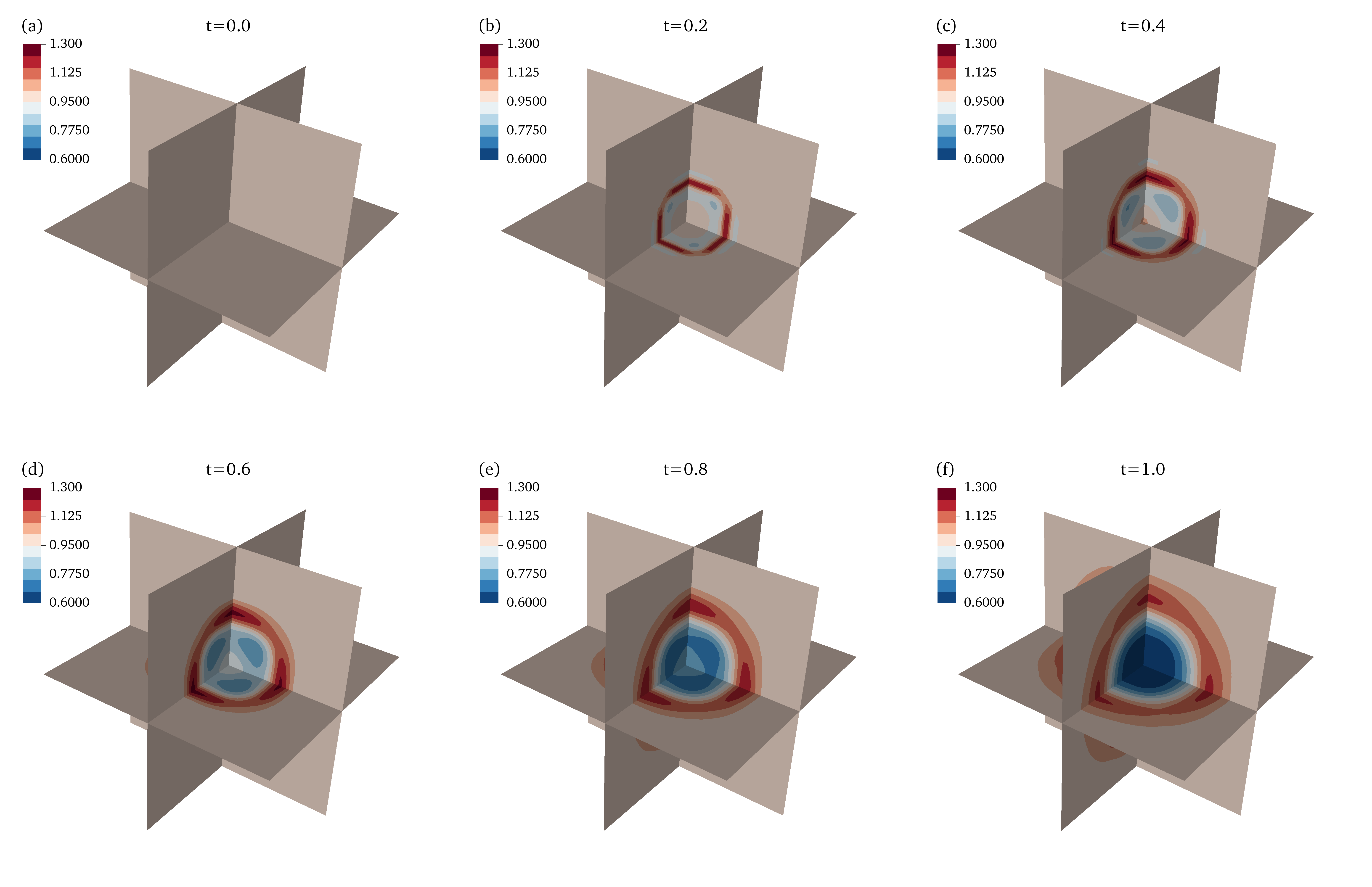}
		\caption{The diffusion problem with discontinuous initial density. Temporal evolution of temperature $T(\bm{x}, t)$.}
		\label{fig:diffusion3d_T}
	\end{figure}

	\section{Conclusion and future work}\label{sec_con}
	To overcome the curse of dimensionality in numerically solving the Boltzmann-BGK equation, the APTT method was developed in this paper.
	After discretizing the Boltzmann-BGK equation by a finite difference scheme, we obtained a tensor algebraic system, whose solution is the PDF and denoted by a full tensor with $m^{2D}$ components.
	To reduce the computational complexity and memory cost, the solutions in the tensor algebraic system were recompressed into low-rank TT format tensors with the lowest TT-rank and satisfying specified accuracy tolerances.
	The nonlinear collision term in the algebraic system was constructed by partially reducing the low-rank TT format tensor of the PDF, resulting in a low-rank TT-based linear system.
	The TT-MALS solver was applied to solve the low-rank TT-based linear system.
	The APTT method reduced the total computational and memory costs by $D-1$ orders of magnitude $m$ compared to traditional methods.
	We demonstrated that the APTT method can maintain the same convergence rate as that of the finite difference scheme and satisfies the conservation laws of mass, momentum, and energy within the prescribed accuracy tolerances.
	Several complex three-dimensional test cases with $m=64$ were performed on a desktop computer, which clearly verifies the efficiency of the newly proposed APTT method.
	There are several possible directions for future work, e.g., applying the APTT method to shock-wave problems and integrating the APTT method into other high-dimensional PDEs.
	
	\section*{Acknowledgments}
	
	C. Xiao has been supported by the Hunan Provincial Natural Science Foundation grant 2023JJ40005 and National Natural Science Foundation of China grant 12131002. K. Tang has been supported by the Hunan Provincial Natural Science Foundation grant 2024JJ6003. J. Huang has been supported by National Natural Science Foundation of China grant 12131002 and 12371439. C. Yang has been supported by National Natural Science Foundation of China grant 12131002.



\bibliographystyle{elsarticle-num}


\bibliography{references.bib}

\begin{thebibliography}{10}
\expandafter\ifx\csname url\endcsname\relax
  \def\url#1{\texttt{#1}}\fi
\expandafter\ifx\csname urlprefix\endcsname\relax\def\urlprefix{URL }\fi
\expandafter\ifx\csname href\endcsname\relax
  \def\href#1#2{#2} \def\path#1{#1}\fi

\bibitem{boyd1995predicting}
I.~D. Boyd, G.~Chen, G.~V. Candler, Predicting failure of the continuum fluid
  equations in transitional hypersonic flows, Physics of Fluids 7~(1) (1995)
  210--219.

\bibitem{sondheimer2001mean}
E.~H. Sondheimer, The mean free path of electrons in metals, Advances in
  Physics 50~(6) (2001) 499--537.

\bibitem{karniadakis2002micro}
G.~E. Karniadakis, A.~Beskok, M.~Gad-el Hak, Micro flows: fundamentals and
  simulation, Applied Mechanics Reviews 55~(4) (2002) B76--B76.

\bibitem{porodnov1974experimental}
B.~Porodnov, P.~Suetin, S.~Borisov, V.~Akinshin, Experimental investigation of
  rarefied gas flow in different channels, Journal of Fluid Mechanics 64~(3)
  (1974) 417--438.

\bibitem{bhatnagar1954model}
P.~L. Bhatnagar, E.~P. Gross, M.~Krook, A model for collision processes in
  gases. {I.} small amplitude processes in charged and neutral one-component
  systems, Physical Review 94~(3) (1954) 511.

\bibitem{de2021physics}
M.~De~Florio, E.~Schiassi, B.~D. Ganapol, R.~Furfaro, Physics-informed neural
  networks for rarefied-gas dynamics: {Thermal} creep flow in the
  {Bhatnagar--Gross--Krook} approximation, Physics of Fluids 33~(4) (2021).

\bibitem{lou2021physics}
Q.~Lou, X.~Meng, G.~E. Karniadakis, Physics-informed neural networks for
  solving forward and inverse flow problems via the {Boltzmann-BGK}
  formulation, Journal of Computational Physics 447 (2021) 110676.

\bibitem{li2024solving}
Z.~Li, Y.~Wang, H.~Liu, Z.~Wang, B.~Dong, Solving {Boltzmann} equation with
  neural sparse representation, SIAM Journal on Scientific Computing 46~(2)
  (2024) C186--C215.

\bibitem{karniadakis2021physics}
G.~E. Karniadakis, I.~G. Kevrekidis, L.~Lu, P.~Perdikaris, S.~Wang, L.~Yang,
  Physics-informed machine learning, Nature Reviews Physics 3~(6) (2021)
  422--440.

\bibitem{hao2022physics}
Z.~Hao, S.~Liu, Y.~Zhang, C.~Ying, Y.~Feng, H.~Su, J.~Zhu, Physics-informed
  machine learning: A survey on problems, methods and applications, arXiv
  preprint arXiv:2211.08064 (2022).

\bibitem{khoromskij2015tensor}
B.~N. Khoromskij, Tensor numerical methods for multidimensional {PDEs}:
  theoretical analysis and initial applications, ESAIM: Proceedings and Surveys
  48 (2015) 1--28.

\bibitem{bachmayr2016tensor}
M.~Bachmayr, R.~Schneider, A.~Uschmajew, Tensor networks and hierarchical
  tensors for the solution of high-dimensional partial differential equations,
  Foundations of Computational Mathematics 16 (2016) 1423--1472.

\bibitem{khoromskij2018tensor}
B.~N. Khoromskij, Tensor numerical methods in scientific computing, Vol.~19,
  Walter de Gruyter GmbH \& Co KG, 2018.

\bibitem{bachmayr2023low}
M.~Bachmayr, Low-rank tensor methods for partial differential equations, Acta
  Numerica 32 (2023) 1--121.

\bibitem{kormann2015semi}
K.~Kormann, A {semi-Lagrangian Vlasov} solver in tensor train format, SIAM
  Journal on Scientific Computing 37~(4) (2015) B613--B632.

\bibitem{ehrlacher2017dynamical}
V.~Ehrlacher, D.~Lombardi, A dynamical adaptive tensor method for the
  {Vlasov--Poisson} system, Journal of Computational Physics 339 (2017)
  285--306.

\bibitem{guo2022low}
W.~Guo, J.-M. Qiu, A low rank tensor representation of linear transport and
  nonlinear {Vlasov} solutions and their associated flow maps, Journal of
  Computational Physics 458 (2022) 111089.

\bibitem{chikitkin2021numerical}
A.~V. Chikitkin, E.~K. Kornev, V.~A. Titarev, Numerical solution of the
  {Boltzmann equation with S-model} collision integral using tensor
  decompositions, Computer Physics Communications 264 (2021) 107954.

\bibitem{boelens2018parallel}
A.~M. Boelens, D.~Venturi, D.~M. Tartakovsky, Parallel tensor methods for
  high-dimensional linear {PDEs}, Journal of Computational Physics 375 (2018)
  519--539.

\bibitem{boelens2020tensor}
A.~M. Boelens, D.~Venturi, D.~M. Tartakovsky, Tensor methods for the
  {Boltzmann-BGK} equation, Journal of Computational Physics 421 (2020) 109744.

\bibitem{dolgov2012fast}
S.~V. Dolgov, B.~N. Khoromskij, I.~V. Oseledets, Fast solution of parabolic
  problems in the tensor train/quantized tensor train format with initial
  application to the {Fokker--Planck} equation, SIAM Journal on Scientific
  Computing 34~(6) (2012) A3016--A3038.

\bibitem{chertkov2021solution}
A.~Chertkov, I.~Oseledets, Solution of the {Fokker--Planck} equation by cross
  approximation method in the tensor train format, Frontiers in Artificial
  Intelligence 4 (2021) 668215.

\bibitem{platkowski1988discrete}
T.~Platkowski, R.~Illner, Discrete velocity models of the {Boltzmann} equation:
  a survey on the mathematical aspects of the theory, SIAM Review 30~(2) (1988)
  213--255.

\bibitem{palczewski1997consistency}
A.~Palczewski, J.~Schneider, A.~V. Bobylev, A consistency result for a
  discrete-velocity model of the {Boltzmann} equation, SIAM Journal on
  Numerical Analysis 34~(5) (1997) 1865--1883.

\bibitem{oseledets2011tensor}
I.~V. Oseledets, Tensor-train decomposition, SIAM Journal on Scientific
  Computing 33~(5) (2011) 2295--2317.

\bibitem{holtz2012alternating}
S.~Holtz, T.~Rohwedder, R.~Schneider, The alternating linear scheme for tensor
  optimization in the tensor train format, SIAM Journal on Scientific Computing
  34~(2) (2012) A683--A713.

\bibitem{bird1970direct}
G.~Bird, Direct simulation and the {Boltzmann} equation, The Physics of Fluids
  13~(11) (1970) 2676--2681.

\bibitem{bird1994molecular}
G.~A. Bird, Molecular gas dynamics and the direct simulation of gas flows,
  Molecular gas dynamics and the direct simulation of gas flows (1994).

\bibitem{rjasanow1996stochastic}
S.~Rjasanow, W.~Wagner, A stochastic weighted particle method for the
  {Boltzmann} equation, Journal of Computational Physics 124~(2) (1996)
  243--253.

\bibitem{rjasanow1998reduction}
S.~Rjasanow, T.~Schreiber, W.~Wagner, Reduction of the number of particles in
  the stochastic weighted particle method for the {Boltzmann} equation, Journal
  of Computational Physics 145~(1) (1998) 382--405.

\bibitem{wu2014solving}
L.~Wu, J.~M. Reese, Y.~Zhang, Solving the {Boltzmann} equation
  deterministically by the fast spectral method: application to gas microflows,
  Journal of Fluid Mechanics 746 (2014) 53--84.

\bibitem{pareschi2000stability}
L.~Pareschi, G.~Russo, On the stability of spectral methods for the homogeneous
  {Boltzmann} equation, Transport Theory and Statistical Physics 29~(3-5)
  (2000) 431--447.

\bibitem{wang2019approximation}
Y.~Wang, Z.~Cai, Approximation of the {Boltzmann} collision operator based on
  hermite spectral method, Journal of Computational Physics 397 (2019) 108815.

\bibitem{panferov2002new}
V.~A. Panferov, A.~G. Heintz, A new consistent discrete-velocity model for the
  {Boltzmann} equation, Mathematical Methods in the Applied Sciences 25~(7)
  (2002) 571--593.

\bibitem{mouhot2006fast}
C.~Mouhot, L.~Pareschi, Fast algorithms for computing the {Boltzmann} collision
  operator, Mathematics of Computation 75~(256) (2006) 1833--1852.

\bibitem{cai2015approximation}
Z.~Cai, M.~Torrilhon, Approximation of the linearized {Boltzmann} collision
  operator for hard-sphere and inverse-power-law models, Journal of
  Computational Physics 295 (2015) 617--643.

\bibitem{gamba2018galerkin}
I.~M. Gamba, S.~Rjasanow, {Galerkin--Petrov} approach for the {Boltzmann}
  equation, Journal of Computational Physics 366 (2018) 341--365.

\bibitem{han2019uniformly}
J.~Han, C.~Ma, Z.~Ma, W.~E, Uniformly accurate machine learning-based
  hydrodynamic models for kinetic equations, Proceedings of the National
  Academy of Sciences 116~(44) (2019) 21983--21991.

\bibitem{huang2022machine}
J.~Huang, Y.~Cheng, A.~J. Christlieb, L.~F. Roberts, Machine learning moment
  closure models for the radiative transfer equation {I}: directly learning a
  gradient based closure, Journal of Computational Physics 453 (2022) 110941.

\bibitem{li2023learning}
Z.~Li, B.~Dong, Y.~Wang, Learning invariance preserving moment closure model
  for {Boltzmann--BGK} equation, Communications in Mathematics and Statistics
  11~(1) (2023) 59--101.

\bibitem{miller2022neural}
S.~T. Miller, N.~V. Roberts, S.~D. Bond, E.~C. Cyr, Neural-network based
  collision operators for the {Boltzmann} equation, Journal of Computational
  Physics 470 (2022) 111541.

\bibitem{schotthofer2022structure}
S.~Schotth{\"o}fer, T.~Xiao, M.~Frank, C.~D. Hauck, Structure preserving neural
  networks: {A} case study in the entropy closure of the {Boltzmann} equation,
  in: Proceedings of the International Conference on Machine Learning, PMLR,
  Baltimore, MD, USA, 2022, pp. 17--23.

\bibitem{xiao2021using}
T.~Xiao, M.~Frank, Using neural networks to accelerate the solution of the
  {Boltzmann} equation, Journal of Computational Physics 443 (2021) 110521.

\bibitem{xiao2023relaxnet}
T.~Xiao, M.~Frank, {RelaxNet:} {A} structure-preserving neural network to
  approximate the {Boltzmann} collision operator, Journal of Computational
  Physics (2023) 112317.

\bibitem{li2022physics}
R.~Li, J.-X. Wang, E.~Lee, T.~Luo, Physics-informed deep learning for solving
  phonon {Boltzmann} transport equation with large temperature non-equilibrium,
  NPJ Computational Materials 8~(1) (2022) 29.

\bibitem{Goodfellow-et-al-2016}
I.~Goodfellow, Y.~Bengio, A.~Courville, Deep Learning, MIT Press, 2016.

\bibitem{cercignani1988boltzmann}
C.~Cercignani, C.~Cercignani, The {Boltzmann} equation, Springer, 1988.

\bibitem{cercignani1997many}
C.~Cercignani, V.~Gerasimenko, D.~Y. Petrina, Many-particle dynamics and
  kinetic equations, Vol. 420, Springer Science \& Business Media, 1997.

\bibitem{dimarco2014numerical}
G.~Dimarco, L.~Pareschi, Numerical methods for kinetic equations, Acta Numerica
  23 (2014) 369--520.

\bibitem{asselin1972frequency}
R.~Asselin, Frequency filter for time integrations, Monthly Weather Review
  100~(6) (1972) 487--490.

\bibitem{thomas2005ncar}
S.~J. Thomas, R.~D. Loft, The {NCAR} spectral element climate dynamical core:
  Semi-implicit {Eulerian} formulation, Journal of Scientific Computing 25~(1)
  (2005) 307--322.

\bibitem{williams2011raw}
P.~D. Williams, The {RAW} filter: {An} improvement to the {Robert--Asselin}
  filter in semi-implicit integrations, Monthly Weather Review 139~(6) (2011)
  1996--2007.

\bibitem{kubacki2013uncoupling}
M.~Kubacki, Uncoupling evolutionary groundwater-surface water flows using the
  {Crank--Nicolson Leapfrog} method, Numerical Methods for Partial Differential
  Equations 29~(4) (2013) 1192--1216.

\bibitem{layton2012stability}
W.~Layton, C.~Trenchea, Stability of two {IMEX} methods, {CNLF} and {BDF2-AB2},
  for uncoupling systems of evolution equations, Applied Numerical Mathematics
  62~(2) (2012) 112--120.

\bibitem{jiang2015crank}
N.~Jiang, M.~Kubacki, W.~Layton, M.~Moraiti, H.~Tran, A {Crank--Nicolson
  Leapfrog} stabilization: {Unconditional} stability and two applications,
  Journal of Computational and Applied Mathematics 281 (2015) 263--276.

\bibitem{kazeev2012low}
V.~A. Kazeev, B.~N. Khoromskij, Low-rank explicit {QTT} representation of the
  {L}aplace operator and its inverse, SIAM Journal on Matrix Analysis and
  Applications 33~(3) (2012) 742--758.

\bibitem{lee2018fundamental}
N.~Lee, A.~Cichocki, Fundamental tensor operations for large-scale data
  analysis using tensor network formats, Multidimensional Systems and Signal
  Processing 29 (2018) 921--960.

\bibitem{strikwerda2004finite}
J.~C. Strikwerda, Finite difference schemes and partial differential equations,
  SIAM, 2004.

\bibitem{saad1986gmres}
Y.~Saad, M.~H. Schultz, {GMRES}: {A} generalized minimal residual algorithm for
  solving nonsymmetric linear systems, SIAM Journal on Scientific and
  Statistical Computing 7~(3) (1986) 856--869.

\bibitem{van1992bi}
H.~A. Van~der Vorst, {Bi-CGSTAB:} a fast and smoothly converging variant of
  {Bi-CG} for the solution of nonsymmetric linear systems, SIAM Journal on
  scientific and Statistical Computing 13~(2) (1992) 631--644.

\bibitem{dolgov2013tt}
S.~V. Dolgov, {TT-GMRES}: solution to a linear system in the structured tensor
  format, Russian Journal of Numerical Analysis and Mathematical Modelling
  28~(2) (2013) 149--172.

\bibitem{ehrlacher2022sott}
V.~Ehrlacher, M.~F. Ruiz, D.~Lombardi, {SoTT}: greedy approximation of a tensor
  as a sum of tensor trains, SIAM Journal on Scientific Computing 44~(2) (2022)
  A664--A688.

\bibitem{al2022parallel}
H.~Al~Daas, G.~Ballard, L.~Manning, Parallel tensor train rounding using {Gram
  SVD}, in: 2022 IEEE International Parallel and Distributed Processing
  Symposium (IPDPS), IEEE, 2022, pp. 930--940.

\bibitem{daas2022parallel}
H.~A. Daas, G.~Ballard, P.~Benner, Parallel algorithms for tensor train
  arithmetic, SIAM Journal on Scientific Computing 44~(1) (2022) C25--C53.

\bibitem{cichocki2016tensor}
A.~Cichocki, N.~Lee, I.~Oseledets, A.-H. Phan, Q.~Zhao, D.~P. Mandic, et~al.,
  Tensor networks for dimensionality reduction and large-scale optimization:
  {Part} 1 low-rank tensor decompositions, Foundations and
  Trends{\textregistered} in Machine Learning 9~(4-5) (2016) 249--429.

\bibitem{espig2020iterative}
M.~Espig, W.~Hackbusch, A.~Litvinenko, H.~G. Matthies, E.~Zander, Iterative
  algorithms for the post-processing of high-dimensional data, Journal of
  Computational Physics 410 (2020) 109396.

\bibitem{oseledets2012solution}
I.~V. Oseledets, S.~V. Dolgov, Solution of linear systems and matrix inversion
  in the {TT-format}, SIAM Journal on Scientific Computing 34~(5) (2012)
  A2718--A2739.

\bibitem{rohrig2023performance}
M.~R{\"o}hrig-Z{\"o}llner, M.~J. Becklas, J.~Thies, A.~Basermann, Performance
  of linear solvers in tensor-train format on current multicore architectures,
  arXiv preprint arXiv:2312.08006 (2023).

\bibitem{duan2017global}
R.~Duan, F.~Huang, Y.~Wang, T.~Yang, Global well-posedness of the {Boltzmann
  equation} with large amplitude initial data, Archive for Rational Mechanics
  and Analysis 225 (2017) 375--424.

\bibitem{duan2019boltzmann}
R.~Duan, Y.~Wang, The {Boltzmann} equation with large-amplitude initial data in
  bounded domains, Advances in Mathematics 343 (2019) 36--109.

\bibitem{bachmayr2012adaptive}
M.~Bachmayr, Adaptive low-rank wavelet methods and applications to two-electron
  schr{\"o}dinger equations, Ph.D. thesis, Hochschulbibliothek der
  Rheinisch-Westf{\"a}lischen Technischen Hochschule Aachen (2012).

\bibitem{bachmayr2015adaptive}
M.~Bachmayr, W.~Dahmen, Adaptive near-optimal rank tensor approximation for
  high-dimensional operator equations, Foundations of Computational Mathematics
  15 (2015) 839--898.

\bibitem{bachmayr2017kolmogorov}
M.~Bachmayr, A.~Cohen, Kolmogorov widths and low-rank approximations of
  parametric elliptic {PDEs}, Mathematics of Computation 86~(304) (2017)
  701--724.

\bibitem{oseledets2011matlab}
I.~Oseledets, S.~Dolgov, et~al., {MATLAB TT-Toolbox version 2.2}, Math Works,
  Natick, MA (2011).

\end{thebibliography}




\end{document}